\documentclass[12pt,reqno]{amsart}

\usepackage[margin=1in]{geometry}
\usepackage{amsmath,amsfonts,amsthm,amssymb,amscd, dsfont,tipa,mathtools,listings, rotating, multirow, bbm}
\usepackage{hyperref}
\usepackage{color}
\usepackage{comment}
\usepackage{seqsplit}
\usepackage{enumitem}
\usepackage{soul}
\usepackage{mathrsfs}
\usepackage[titletoc,title]{appendix}

\usepackage{constants}
\newconstantfamily{abcon}{symbol=c}

\newcommand{\bburl}[1]{\textcolor{blue}{\url{#1}}}

\newcommand\Item[1][]{%
  \ifx\relax#1\relax  \item \else \item[#1] \fi\abovedisplayskip=0pt\abovedisplayshortskip=0pt~\vspace*{-\baselineskip}}

\newcommand\be{\begin{equation}}
	\newcommand\ee{\end{equation}}
\newcommand\bi{\begin{itemize}}
	\newcommand\ei{\end{itemize}}
\newcommand\ben{\begin{enumerate}}
	\newcommand\een{\end{enumerate}}

\newtheorem{thm}{Theorem}[section]

\newtheorem{cor}[thm]{Corollary}
\newtheorem{lem}[thm]{Lemma}
\newtheorem{prop}[thm]{Proposition}

\theoremstyle{remark}
\newtheorem*{rek}{Remark}

\newcommand{\z}{\mathbb{Z}}
\newcommand{\h}{\mathbb{H}}

\newcommand{\bs}{\backslash}
\DeclareMathOperator{\SL}{SL}
\DeclareMathOperator{\GL}{GL}
\DeclareMathOperator{\ad}{Ad}

\DeclareMathOperator{\Real}{Re}
\DeclareMathOperator{\Imag}{Im}
\DeclareMathOperator{\supp}{supp}
\DeclareMathOperator{\ord}{ord}
\DeclareMathOperator{\res}{res}

\DeclareMathOperator{\new}{new}

\makeatletter
\let\@wraptoccontribs\wraptoccontribs
\makeatother

\numberwithin{equation}{section}

\usepackage{lineno}
\usepackage{lipsum}

\begin{document}

\title[Effective correlation/decorrelation and weak subconvexity]{Effective correlation and decorrelation for newforms, and weak subconvexity for $L$-functions}

\author[Nawapan Wattanawanichkul]{Nawapan Wattanawanichkul \\ (with an appendix by Jesse Thorner)}
\thanks{Address: 1409 West Green Street, University of Illinois Urbana-Champaign, Urbana, IL 61801, USA}
\thanks{Email: nawapanwattanawanichkul4@gmail.com (Nawapan Wattanawanichkul)}
\thanks{Email: jesse.thorner@gmail.com (Jesse Thorner)}
\thanks{Competing Interests: The authors N.W. and J.A.T. declare that they have no competing interests.}

\begin{abstract}
    Let $f$ and $g$ be spectrally normalized holomorphic newforms of even weight $k \geq2$ on $\Gamma_0(q)$. 
    If $f\neq g$, then assume that $q$ is squarefree.  For a nice test function $\psi$ supported on $\Gamma_0(1)\backslash\mathbb{H}$, we establish the best known bound (uniform in $k$, $q$, and $\psi$) for
    \[ 
        \int_{\Gamma_0(q)\backslash\mathbb{H}}\psi(z)f(z)\overline{g(z)}y^{k}\frac{dxdy}{y^2}-\mathds{1}_{f = g}\frac{3}{\pi}\int_{\Gamma_0(1)\backslash\mathbb{H}}\psi(z)\frac{dx dy}{y^2}.
    \]
    When $f=g$, our result yields an effective holomorphic variant of quantum unique ergodicity, refining work of Holowinsky--Soundararajan and Nelson--Pitale--Saha. 
    When $f \neq g$, our result extends and improves the effective decorrelation result of Huang for $q=1$. 
    To prove our results, we refine Soundararajan's weak subconvexity bound for Rankin--Selberg $L$-functions.
\end{abstract}



\maketitle


\vspace{-3mm}

\section{Introduction}

\label{sec:intro}

    Let $\Gamma_0(1) = \SL_2(\z)$, and consider the hyperbolic surface $Y_0(1) = \Gamma_0(1)\bs\h$. The volume form is $d\mu = y^{-2}dxdy$, and $\mu(Y_0(1))=\pi/3$. Given bounded measurable functions $H_1$ and $H_2$ on $Y_0(1)$, their Petersson inner product is given by
    \[
        \langle H_1, H_2 \rangle = \int_{Y_0(1)} H_1(z) \overline{H_2(z)} d\mu.
    \]
    
    In this paper, we focus on a holomorphic variant of  quantum unique ergodicity (QUE).  
    Let $z= x+iy$, and let 
    \[
        f(z) = \sum_{n=1}^{\infty}\lambda_f(n)n^{\frac{k-1}{2}}e^{2\pi i n z} = e^{2\pi i z}+\sum_{n=2}^{\infty}\lambda_f(n)n^{\frac{k-1}{2}}e^{2\pi i n z}
    \]
    be a normalized Hecke eigenform of even weight $k\geq 2$ and level $1$. In other words, $f$ is a cusp form of weight $k$ and level $1$ that is also an eigenfunction of all the Hecke operators with $\lambda_f(1) = 1$.  
    Consider a sequence of such $f$ with $k\to\infty$. 
    Write $F_k(z) = \rho_f(1)y^{k/2}f(z)$, where $\rho_f(1)$ is chosen such that $\langle 1, |F_k|^2 \rangle = 1$. The holomorphic QUE conjecture asserts that if $\psi$ is a bounded measurable function on  $Y_0(1)$, then
    \begin{equation}
	\label{eqn:HS_qualitative}
            \lim_{k \to \infty}   \langle \psi,|F_k|^2 \rangle = \frac{3}{\pi} \langle \psi, 1\rangle.
    \end{equation}
    Holowinsky and Soundararajan \cite{H,HS,S}  proved that if $\varepsilon>0$, then
    \begin{equation}
	\label{eqn:HolowinskySound}
            \langle \psi, |F_k|^2 \rangle = \frac{3}{\pi}\langle \psi, 1\rangle +O_{\psi, \varepsilon}((\log k)^{-\frac{1}{30}+\varepsilon}).
    \end{equation}
    
    The proof of \eqref{eqn:HolowinskySound} in \cite{HS} relies on two different approaches.  In \cite{S}, Soundararajan spectrally expands the left-hand side of \eqref{eqn:HolowinskySound} in terms of Hecke--Maa{\ss} cusp forms and unitary Eisenstein series, bounding the error term using a weak subconvexity bound (see \eqref{eq:sound-weaksubconvexity}).  In \cite{H}, Holowinsky expands the left-hand side of \eqref{eqn:HolowinskySound} in terms of incomplete Poincar{\'e} series and 
    incomplete Eisenstein series, bounding the error term by analyzing shifted convolution sums that arise from the expansion.  A rate of convergence of $O_{\psi,\varepsilon}(k^{-1/2+\varepsilon})$ follows from the generalized Lindel{\"o}f hypothesis; up to refinements in the $k^{\varepsilon}$, this is optimal (see \cite{IS}). Using \eqref{eqn:HS_qualitative}, Rudnick \cite{Ru} proved that as $k\to\infty$, the zeros of $f$ equidistribute on $Y_0(1)$.
	
    Let $C_c^{\infty}(Y_0(1))$ be the space of compactly supported, infinitely differentiable functions on $Y_0(1)$. Let $\mathbb{N}$ denote the set of positive integers.  Given $M\geq 1$, let $C_{c}^{\infty}(Y_0(1),M)$ equal
    \begin{equation}
    \label{eq:M-norm}
       \hspace{-1mm}
       \Big\{\psi\in C_{c}^{\infty}(Y_0(1))\colon \textup{if $a,b \in \mathbb{N}\cup\{0\}$, then } \sup_{z \in Y_0(1)} \Big|y^{\max(a+b,\frac{1}{2})}\frac{\partial^a}{\partial x^a}\frac{\partial^b}{\partial y^b} \psi(z)\Big| \ll_{a,b} M^{a+b} \Big\}.
    \end{equation}
    Let $\mathcal{F}$ be a fundamental domain of $Y_0(1)$.  A close inspection of the work of Lester, Matom{\"a}ki, and Radziwi{\l}{\l} \cite[Proof of Theorem 1.1]{LMR} shows that if $M\geq 1$, $\psi\in C_c^{\infty}(Y_0(1),M)$, and there exist absolute and effectively computable constants $A>0$ and $\delta>0$ such that
    \begin{equation}
        \label{eqn:LMR}
            \langle \psi, |F_k|^2\rangle = \frac{3}{\pi}\langle \psi,1\rangle+ O_{\varepsilon}(M^{A+\varepsilon}(\log k)^{-\delta+\varepsilon}),
    \end{equation}
    then for any fixed $B>0$ and any hyperbolic ball ${\mathscr{B}}(z_0,r)\subseteq\{z\in\mathcal{F}\colon \mathrm{Im}(z)\leq B\}$ centered at $z_0$ with radius $r>0$, the following effective variant of Rudnick's result holds:
    \begin{equation}
	\label{eqn:LMR_zeros}
            \frac{\#\{\varrho_f \in \mathscr{B}(z_0,r)\colon f(\varrho_f) = 0\}}{\#\{\varrho_f \in \mathcal{F}\colon f(\varrho_f) = 0\}} = \frac{3}{\pi}\mu(\mathscr{B}(z_0,r))+ O_{B,\varepsilon}\Big(\frac{r}{(\log k)^{\frac{\delta}{4+2A}-\varepsilon}}\Big).
    \end{equation}
    Using Iwaniec’s course notes \cite{Iwaniec}, Lester, Matom\"aki, and Radziwi{\l\l} \cite[Theorem 1.3]{LMR} proved that \eqref{eqn:LMR} holds with $A=2$ and $\delta=31/2-4\sqrt{15} = 0.008066\ldots$. Iwaniec's course notes present a proof of holomorphic QUE that does not require Soundararajan's weak subconvexity result from \cite{S}. However, upon closer examination, we identified a missing factor in the final estimate of \cite[Section 11]{Iwaniec}, which also influences the proof of \cite[Theorem 1.3]{LMR}. Accounting for this factor yields a convergence rate that is slower than what was stated in the original argument. See Appendix \ref{appendix} for details. Alternatively, \cite[Theorem 1.3]{LMR} can be proved by refining Soundararajan's weak subconvexity bound.  We discuss this below.
	
    Kowalski, Michel, and Vanderkam \cite[Conjecture 1.4]{KMV} introduced a variant of holomorphic QUE in which the levels of holomorphic newforms are allowed to vary.  For $z=x+iy$, let
    \[
        f(z) = \sum_{n=1}^{\infty}\lambda_f(n)n^{\frac{k-1}{2}}e^{2\pi i n z}\in S_k^{\mathrm{new}}(\Gamma_0(q))
    \]
    be a holomorphic newform of even weight $k\geq 2$ and level $q\geq 1$ with trivial nebentypus (see Section \ref{subsec:newforms} for more details).  We consider a sequence of such $f$ with $kq\to\infty$. Given bounded measurable functions $H_1$ and $H_2$ on  $Y_0(q)=\Gamma_0(q)\bs\h$, we define 
    \begin{equation}
        \label{eq:Petersson}
     \langle H_1, H_2\rangle_q = \int_{Y_0(q)}H_1(z)\overline{H_2(z)} d\mu.
    \end{equation}
    Write
    \begin{equation}
    \label{eq:F_k(z)}
        F_k(z) = \rho_f(1)y^{k/2}f(z),
    \end{equation}
    where $\rho_f(1)$ is chosen such that $ \langle 1, |F_k(z)|^2 \rangle_q = 1$. 
    Kowalski, Michel, and Vanderkam conjectured that if $\psi$ is a bounded measurable function supported on $Y_0(1)$, then
    \begin{equation}
    \label{eq:hQUEstatement}
        \lim_{kq \to \infty}\langle \psi, |F_k|^2 \rangle_q = \frac{3}{\pi}\langle \psi, 1\rangle_1.
    \end{equation}
    Building on the work in \cite{H,HS,S}, Nelson, Pitale, and Saha \cite{N,NPS} proved that there exist constants $\delta_1,\delta_2>0$ such that if $q_0$ is the largest squarefree divisor of $q$, then
    \[
      \langle \psi, |F_k|^2 \rangle_q = \frac{3}{\pi} \langle \psi, 1 \rangle_1 + O_{\psi}\Big(\Big(\frac{q}{q_0}\Big)^{-\delta_1}(\log kq)^{-\delta_2}\Big).
    \]
    Thus, the rate of convergence is sensitive to whether the level $q$ is powerful.

    Let $f$ and $g$ be holomorphic newforms of weight $k$ and level $q$, and let $G_k$ be defined similarly to \eqref{eq:F_k(z)}.
     In analogy with the variations on holomorphic QUE mentioned above, it is natural to conjecture that if $\psi$ is a bounded measurable function supported on $Y_0(1)$, then
    \begin{equation}
        \label{eq:decorrelation}
            \lim_{kq \to \infty} |\langle \psi, \overline{F_k}G_k \rangle_q| = 0.
    \end{equation}
    Huang \cite{H} proved that if $f$ and $g$ are normalized Hecke eigenforms of weight $k$ and level $1$, with $f\neq g$, $M\geq 1$, and $\psi \in C_{c}^{\infty}(Y_0(1),M)$, then for all $\varepsilon>0$, we have
    \begin{equation}
        \label{eq:Huang-decor}
             \langle \psi, \overline{F_k}G_k \rangle_1  \ll_\varepsilon M^{5/3}(\log k)^{-1.19\times 10^{-41} +\varepsilon}.
    \end{equation}

    Now, consider holomorphic newforms $f$ and $g$ of weight $k$ and level $q$.  If $f \neq g$, then we assume that $q$ is squarefree.  In this paper, we establish the strongest known unconditional upper bound on $|\langle \psi, \overline{F_k}G_k \rangle_q - \mathds{1}_{f=g}\frac{3}{\pi}\langle \psi, 1 \rangle_1|$.  When $f=g$, this addresses the rate of convergence in \eqref{eq:hQUEstatement} with effective dependence on $k$, $q$, and $\psi$. When $f\neq g$, this addresses the rate of convergence in \eqref{eq:Huang-decor} and extends such result to squarefree levels.  For $f=g$ and $q = 1$, our result provides a faster rate of convergence than in \eqref{eqn:HolowinskySound} and \cite[Theorem 1.3]{LMR}.

    \begin{thm}
    \label{thm:cor}
        Let $q\geq 1$ be an integer, and let $q_0$ be the largest squarefree divisor of $q$. Let $k\geq 2$ be an even integer.  Let $f$ and $g$ be holomorphic newforms of weight $k$, level $q$, and trivial nebentypus.  Let $F_k$ and $G_k$ be defined as in \eqref{eq:F_k(z)}. Assume that if $f\neq g$, then $q$ is squarefree.
        Let $\theta\in[0,7/64]$ denote the best bound towards the generalized Ramanujan conjecture for Hecke--Maa{\ss} cusp forms. If $\varepsilon>0$, $M\geq 1$, and $\psi\in C_c^{\infty}(Y_0(1),M)$, then  
	\[
       \langle \psi, \overline{F_k}G_k \rangle_q - \mathds{1}_{f=g}\frac{3}{\pi}\langle \psi, 1 \rangle_1\ll_{\varepsilon} M^{4} \Big(\frac{q}{q_0}\Big)^{-\frac{1}{12}(2\sqrt{3}-3)(1-2\theta)+\varepsilon} (\log kq)^{-(\frac{7}{2}-2\sqrt{3})+\varepsilon}.
	\]
    If $q=1$ and $f=g$, then we can replace $M^4$ with $M^{(23-2\sqrt{3})/12}$. 
    \end{thm}

    \begin{cor}
    Let $f$ be a normalized Hecke eigenform of even weight $k\geq 2$ and level $1$.  Equation \eqref{eqn:LMR_zeros} holds with $A=\frac{1}{12}(23-2\sqrt{3}) = 1.627991\ldots$ and $\delta = \frac{7}{2}-2\sqrt{3} =0.035898\ldots$.  In particular, if $B>0$ is fixed and $\mathcal{F}$ is a fundamental domain of $Y_0(1)$, then the zeros of $f$ equidistribute with respect to $d\mu$ within any hyperbolic ball ${\mathscr{B}}(z_0,r)\subseteq\{z\in\mathcal{F}\colon \mathrm{Im}(z)\leq B\}$ of radius $r\geq (\log k)^{-1/203}$.
    \end{cor}
    
    The main new result that enables us to prove Theorem \ref{thm:cor}  is a refinement of Soundararajan's weak subconvexity bound. For this refinement, let $m \ge 1$ be an integer, and let $\mathfrak{F}_{m}$ be the set of all cuspidal automorphic representations of $\GL_m(\mathbb{A}_{\mathbb{Q}})$ with unitary central character. We normalize each $\pi\in\mathfrak{F}_m$ so that the central character is  trivial on the diagonally embedded copy of the positive reals. Consider $\pi \in \mathfrak{F}_{m}$ and $\pi' \in \mathfrak{F}_{m'}$.  Let $\mathfrak{C}(\pi\times\pi')$ denote the analytic conductor of the Rankin--Selberg $L$-function $L(s,\pi\times\pi')$ (see Section \ref{subsec:RS}). 
    Soundararajan \cite{S} proved the following {\it weak subconvexity bound}:  If $\pi$ satisfies the generalized Ramanujan conjecture (GRC), then for any $\varepsilon > 0$,
    \begin{equation}
	\label{eq:sound-weaksubconvexity}
            L(1/2, \pi \times \pi') \ll_{ m, \pi', \varepsilon} \frac{\mathfrak{C}(\pi\times \pi')^{1/4}}{[\log \mathfrak{C}(\pi\times \pi')]^{1-\varepsilon}}.
    \end{equation}
    If $\pi$ satisfies GRC, then \eqref{eq:sound-weaksubconvexity} provides a logarithmic improvement over the convexity bound
    when $\pi'$ is fixed. To obtain the explicit dependence on $\psi$ as presented in Theorem \ref{thm:cor}, one must explicate the dependence on $\pi'$  in \eqref{eq:sound-weaksubconvexity}. Alternatively, when $\pi$ satisfies GRC, one might use the bound
	\[
            L(1/2,\pi\times\pi')\ll_{m,m'}\frac{\mathfrak{C}(\pi\times\pi')^{1/4}}{[\log \mathfrak{C}(\pi\times\pi')]^{1/(10^{17}(mm')^3)}},
	\]
    proved by Soundararajan and Thorner \cite{ST}.  While this bound is fully uniform in both $\pi$ and $\pi'$, the power of $\log \mathfrak{C}(\pi\times\pi')$ that it saves is quite small.  In this paper, we prove the following refinement of \eqref{eq:sound-weaksubconvexity}, which we expect to be useful in settings beyond what we consider here.
	
    \begin{thm}
    \label{thm:weaksub}
        Let $\pi \in \mathfrak{F}_{m}$ and $\pi' \in \mathfrak{F}_{m'}$. Let $L(s,\pi\times\pi')$ be entire and $\pi$ satisfy GRC.\footnote{We only require $\pi$ to satisfy GRC at the non-archimedean places in our proof.}
        If $0<\varepsilon<1$, then 
		\begin{equation*}
                L(1/2, \pi\times\pi') \ll_{m,m',\varepsilon} \mathfrak{C}(\pi')^{\varepsilon} \frac{\mathfrak{C}(\pi\times\pi')^{1/4}}{[\log \mathfrak{C}(\pi\times\pi')]^{1-\varepsilon}}.
		\end{equation*}
    \end{thm}
    \begin{rek}
        Our proof shows that there is an effectively computable constant $\Cl[abcon]{Li-thm1.3}=\Cr{Li-thm1.3}(m,m')>0$ such that the $\mathfrak{C}(\pi')^{\varepsilon}$ in Theorem \ref{thm:weaksub} can be refined to $\exp[\Cr{Li-thm1.3}(\log \mathfrak{C}(\pi'\times\widetilde{\pi}'))/\log\log \mathfrak{C}(\pi'\times\widetilde{\pi}')]$.
    \end{rek}

    \subsection*{Acknowledgements} 
    The author thanks  Jesse Thorner for providing invaluable feedback and guidance throughout this project. The author also thanks Peter Humphries, Maksym Radziwi\l\l, and Bingrong Huang for helpful  comments on the manuscript. Finally, the author would like to thank the referee for their careful reading and valuable suggestions.
   

    \section{Outline of the paper}

    In Section \ref{sec:L-functions}, we record some properties of standard $L$-functions and Rankin--Selberg $L$-functions, and establish useful bounds for partial sums of Dirichlet coefficients. 
    
    In Section \ref{sec:weak-subconvexity-bound}, we prove Theorem \ref{thm:weaksub}.

    In Section \ref{sec:GL(2)}, we provide background on $\GL(2)$ automorphic forms, including holomorphic newforms, Hecke--Maa{\ss} cusp forms, Eisenstein series, incomplete Eisenstein series, and incomplete Poincar\'e series, and recall the spectral decomposition and the Poincar\'e decomposition.

    In Section \ref{sec:outline-thm1.1-1.2}, we prove Theorem \ref{thm:cor} assuming Propositions  \ref{prop:Sound_cor} and \ref{prop:Holow_cor}.

    In Section \ref{sec:Bound-L-functions}, we establish useful lower bounds for the adjoint lift of a holomorphic newform. We also prove bounds for $L$-functions that will be useful for proving Propositions \ref{prop:Sound_cor}.

    In Section \ref{sec:correlation}, we prove Propositions \ref{prop:Sound_cor} and \ref{prop:Holow_cor}. 

    \section{Properties of \texorpdfstring{$L$}{L}-functions}
    \label{sec:L-functions}

    We use the standard notation $f(z) \asymp g(z)$, $f(z) \ll g(z)$, and $f(z) = O(g(z))$ from analytic number theory. In particular,
    for any parameter $\nu$, if
    we write $f(z) = O_{\nu}(g(z))$ or $f(z) \ll_{\nu} g(z)$, then there is an effectively computable constant $c>0$, depending at most on $m$, $m'$, $\varepsilon$, and $\nu$, such that $|f(z)| \le c|g(z)|$ for all $z$ under consideration. If $\nu$ is absent, then $c$ depends at most on $m$, $m'$, and $\varepsilon$. Finally, we write $f(z) \asymp g(z)$ if both $f(z) \ll g(z)$ and $f(z) \gg g(z)$ hold.

    \subsection{Standard \texorpdfstring{$L$}{L}-functions}
    \label{subsec: standardL}
    
    Let $m \ge 1$ be an integer, $\mathbb{A}_{\mathbb{Q}}$ denote the ring of ad\`{e}les over $\mathbb{Q}$, and $\mathfrak{F}_m$ be the set of all cuspidal automorphic representations of $\mathrm{GL}_m(\mathbb{A}_{\mathbb{Q}})$ with central characters normalized to be trivial on the positive reals.  
    If $\pi\in\mathfrak{F}_m$, then for any $v$ of $\mathbb{Q}$, there exists a smooth admissible representation $\pi_v$ of $\GL_m(\mathbb{Q}_v)$ such that  $\pi$ can be written as the restricted tensor product $\otimes_v \pi_v$.  
    For non-archimedean $v$ corresponding with a prime $p$, the local $L$-function $L(s,\pi_{p})$ is defined by the Satake parameters $\alpha_{1,\pi}(p),\ldots,\alpha_{m,\pi}(p)$ as
    \begin{equation}
	\label{eqn:Euler_p_single}
        L_p(s,\pi)=\prod_{j=1}^{m}(1-\alpha_{j,\pi}(p)p^{-s})^{-1}=\sum_{k=0}^{\infty}\frac{\lambda_{\pi}(p^k)}{p^{ks}}.
    \end{equation}
     Let $N_{\pi}$ be the conductor of $\pi$. 
    If $p\nmid N_{\pi}$, then for all $j$, we have that $\alpha_{j,\pi}(p)\neq0$.  If $p\mid N_{\pi}$, it might be the case that there exists $j$ such that $\alpha_{j,\pi}(p)=0$.  The standard $L$-function $L(s,\pi)$ associated to $\pi$ is
    \[
        L(s,\pi)=\prod_{p} L_{p}(s,\pi)=\sum_{n=1}^{\infty}\frac{\lambda_{\pi}(n)}{n^s}.
    \]
    The Euler product and Dirichlet series converge absolutely when $\mathrm{Re}(s)>1$.
	
    At the archimedean place of $\mathbb{Q}$, there are $m$ Langlands parameters $\mu_{\pi}(j)\in\mathbb{C}$ such that
    \[
        L_{\infty}(s,\pi) = \pi^{-\frac{ms}{2}}\prod_{j=1}^{m}\Gamma\Big(\frac{s+\mu_{\pi}(j)}{2}\Big).
    \]
    Luo, Rudnick, and Sarnak \cite{LRS} and M\"uller  and Speh \cite{MS} proved that there exists $\theta_m\in[0,\frac{1}{2}-\frac{1}{m^2+1}]$ such that we have the uniform bounds
    \begin{equation}
	\label{eqn:LRS_finite}
            |\alpha_{j,\pi}(p)|\leq  p^{\theta_m}\qquad\textup{ and }\qquad \Real(\mu_{\pi}(j)) \geq -\theta_m,
    \end{equation}
    and GRC asserts that one may take $\theta_m=0$. By the work of Kim and Sarnak \cite[Appendix 2]{Kim} and Blomer and Brumley \cite{BB}, we have the bound $\theta_2 \le 7/64$. Let $\widetilde{\pi}\in\mathfrak{F}_m$ be the contragredient representation of $\pi$. We have $N_{\pi}=N_{\widetilde{\pi}}$, and we have the equalities of sets
    \[
    \{\alpha_{j,\widetilde{\pi}}(p)\}=\{\overline{ \alpha_{j,\pi}(p)}\} \quad \text{and} \quad \{\mu_{\widetilde{\pi}}( j)\}=\{\overline{\mu_{\pi}(j)}\}.
    \]
	
    Let $r_{\pi}$ be the order of the pole of $L(s,\pi)$ at $s=1$.  The completed $L$-function
    \[
        \Lambda(s,\pi) = (s(1-s))^{r_{\pi}}N_{\pi}^{s/2}L(s,\pi)L_{\infty}(s,\pi)
    \]
    is entire of order 1. There exists $W(\pi) \in \mathbb{C}$ of modulus 1 such that for all $s\in\mathbb{C}$, we have the functional equation $\Lambda(s,\pi)=W(\pi)\Lambda(1-s,\widetilde{\pi})$. The analytic conductor of $\pi$ \cite{IS} is given by
    \begin{equation}
	\label{eqn:analytic_conductor_def}
            \mathfrak{C}(\pi,t)\coloneqq N_{\pi}\prod_{j=1}^m (3+|it+\mu_{\pi}(j)|),\qquad \mathfrak{C}(\pi)\coloneqq \mathfrak{C}(\pi,0).
    \end{equation}
    We define $a_{\pi}(p^k)$ by the Dirichlet series identity
    \[
    -\frac{L'}{L}(s, \pi) = \sum_{p}\sum_{k=1}^{\infty} \frac{\sum_{j=1}^m \alpha_{j,\pi}(p)^k\log p}{p^{ks}} = \sum_{p}\sum_{k=1}^{\infty} \frac{a_{\pi}(p^k)\log p}{p^{ks}},  \quad \Real(s) > 1. 
    \]

    \subsection{Rankin--Selberg \texorpdfstring{$L$}{L}-functions}
    \label{subsec:RS}

    Let $\pi \in \mathfrak{F}_m$ and $\pi' \in \mathfrak{F}_{m'}$.  For each $p \nmid N_{\pi}N_{\pi'}$, define
    \[
        L_p(s, \pi \times \pi')= \prod_{j=1}^{m} \prod_{j'=1}^{m'} \frac{1}{1-\alpha_{j,\pi}(p)\alpha_{j',\pi'}(p)p^{-s}}.
    \]
    Jacquet, Piatetski-Shapiro, and Shalika \cite{JPS} proved the following theorem. 

    \begin{thm}
    \label{thm:Rankin-Selberg}
    If $(\pi, \pi') \in \mathfrak{F}_m \times \mathfrak{F}_{m'}$, then there exist
    \begin{enumerate}
        \item complex numbers $(\alpha_{j,j', \pi \times \pi'}(p))_{j=1}^m{}_{j'=1}^{m'}$ for each prime $p\,|\, N_{\pi}N_{\pi'}$, from which we define 
        \[
            L_{p}(s,\pi\times\pi')=\prod_{j=1}^{m}\prod_{j'=1}^{m'}\frac{1}{1-\alpha_{j,j',\pi\times\pi'}(p) p^{-s}}, \quad L_{p}(s,\widetilde{\pi}\times\widetilde{\pi}')=\prod_{j=1}^{m}\prod_{j'=1}^{m'}\frac{1}{1-\overline{\alpha_{j,j',\pi\times\pi'}(p)} p^{-s}};
        \]
        \item  complex numbers $(\mu_{ \pi \times \pi'}(j,j'))_{j=1}^{m}{}_{j'=1}^{m'}$, from which we define 
        \begin{equation*}
        \begin{aligned}
            &L_{\infty}(s,\pi\times\pi') = \pi^{-\frac{mm's}{2}}\prod_{j=1}^{m}\prod_{j'=1}^{m'}\Gamma\Big(\frac{s+\mu_{\pi\times\pi'}(j,j')}{2}\Big),\\
            &L_{\infty}(s,\widetilde{\pi}\times\widetilde{\pi}') = \pi^{-\frac{mm's}{2}}\prod_{j=1}^{m}\prod_{j'=1}^{m'}\Gamma\Big(\frac{s+\overline{\mu_{\pi\times\pi'}(j,j')}}{2}\Big);    
        \end{aligned}   
        \end{equation*} 
    \item an arithmetic conductor, denoted as $N_{\pi \times \pi'}$; and 
    \item a complex number $W(\pi \times \pi')$ of modulus 1
    \end{enumerate}
    such that the Rankin--Selberg $L$-function
    \[   L(s,\pi\times\pi')=\prod_{p}L_{p}(s,\pi\times\pi') =\sum_{n=1}^{\infty}\frac{\lambda_{\pi\times\pi'}(n)}{n^s}
    \]
    converges absolutely for $\Real(s) > 1$ with a pole at $s=1$ of order $r_{\pi\times\pi'} = 1$ if  $\pi'=\widetilde{\pi}$, or $r_{\pi\times\pi'} =0$ otherwise,
    the completed $L$-function
    \begin{equation}
    \label{eqn:Lambdaspixpi'}
        \Lambda(s,\pi\times\pi')=(s(1-s))^{r_{\pi\times\pi'}}N_{\pi\times\pi'}^{s/2}L(s,\pi\times\pi')L_{\infty}(s,\pi\times\pi')
    \end{equation}
    is entire of order 1, and $\Lambda(s,\pi\times\pi')= W(\pi\times\pi')\Lambda(1-s,\widetilde{\pi}\times\widetilde{\pi}')$.
    \end{thm}

   To determine $\mu_{ \pi \times \pi'}(j,j')$ (resp. $\alpha_{j,j', \pi \times \pi'}(p)$ at primes $p\,|\, N_{\pi}N_{\pi'}$), we use the archimedean case of the local Langlands correspondence \cite{MS} (resp. \cite[Appendix]{ST}) and \eqref{eqn:LRS_finite} and obtain
    \begin{equation}
	\label{eqn:LRS_2}	    |\alpha_{j,j',\pi\times\pi'}(p)|\leq p^{\theta_m + \theta_{m'}},\qquad \Real(\mu_{\pi\times\pi'}(j,j'))\geq -(\theta_m + \theta_{m'}).
    \end{equation}
    If $ k \ge 1$ is an integer, then we define 
    \begin{equation}
        a_{\pi \times \pi'}(p^k) = \begin{cases}
            a_\pi(p^k)a_{\pi'}(p^k) & \text{if } p\nmid N_{\pi}N_{\pi'},\\
            \sum_{j=1}^{m} \sum_{j'=1}^{m'} \alpha_{j,j',\pi \times \pi'}(p)^k & \text{otherwise.}
        \end{cases}
    \end{equation}
    We define $\Lambda_{\pi\times\pi'}(n)$ by the Dirichlet series identity
    \begin{equation}
        -\frac{L'}{L}(s,\pi \times \pi') = \sum_{p}\sum_{k=1}^{\infty} \frac{a_{\pi \times \pi'}(p^k)\log p}{p^{ks}} = \sum_{n=1}^{\infty} \frac{\Lambda_{\pi\times\pi'}(n)}{n^{s}}, \quad \Real(s) > 1.
    \end{equation}

    As with $L(s,\pi)$, we  define the analytic conductor
\begin{equation}\label{eqn:analytic_conductor_def_2}
	\begin{aligned}
        \mathfrak{C}(\pi\times\pi',t) \coloneqq N_{\pi\times\pi'}\prod_{j=1}^m \prod_{j'=1}^{m'}(3+|it+\mu_{\pi\times\pi'}(j,j')|),\quad 
        \mathfrak{C}(\pi\times\pi') \coloneqq \mathfrak{C}(\pi\times\pi',0).
	\end{aligned}
    \end{equation}
    The combined work of Bushnell and Henniart \cite{BushHen} and Brumley \cite[Appendix]{Humphries} yields
    \begin{equation}
	\label{eqn:BH}
        \mathfrak{C}(\pi\times\pi',t)\ll \mathfrak{C}(\pi\times\pi')(3+|t|)^{mm'},\quad \mathfrak{C}(\pi\times\pi')\ll \mathfrak{C}(\pi)^{m'} \mathfrak{C}(\pi')^{m}.
    \end{equation}
    
    \subsection{Bounds for partial sums of Dirichlet coefficients} 
    Consider $\pi \in \mathfrak{F}_m$ and $\pi' \in \mathfrak{F}_{m'}$. Brumley \cite[Appendix]{ST} and Jiang, L\"u, and Wang \cite{JLW} proved the following useful inequalities.
	
    \begin{lem}[Lemma 2.2 of \cite{ST}]
        \label{lem:lem2.2-ST} 
        Let $\pi \in \mathfrak{F}_m$ and $\pi' \in \mathfrak{F}_{m'}$.  If $n \ge 1$ is an integer, then
	\begin{equation*}
            |\Lambda_{\pi \times \pi'}(n)| \le \sqrt{\Lambda_{\pi \times \widetilde{\pi}}(n) \Lambda_{\pi' \times \widetilde{\pi}'}(n)}.
	\end{equation*}
    \end{lem}

    \begin{lem}[Lemma 3.1 of \cite{JLW}]
        \label{lem:lem3.1-JLW}
        Let $\pi \in \mathfrak{F}_m$ and $\pi' \in \mathfrak{F}_m$.  If $n\geq 1$ is an integer, then
        \begin{equation*}
            |\lambda_{\pi\times\pi'}(n)| \le \sqrt{\lambda_{\pi \times \widetilde{\pi}}(n)\lambda_{\pi'\times\widetilde{\pi}'}(n)}.
	\end{equation*}
    \end{lem}
  	
    Lemmas \ref{lem:lem2.2-ST} and \ref{lem:lem3.1-JLW} are useful to us because we assume that $\pi$ satisfies GRC.  To handle the contribution from $\pi'$, we use two results, the first of which is due to Li.
	
    \begin{thm}[Theorem 2 of \cite{L}]
        \label{thm:Xiannan-Li} 
        There exists an absolute and effectively computable constant $\Cl[abcon]{Li} >0$ such that if $\pi' \in \mathfrak{F}_{m'}$ and $1 < \sigma \le 3$, then 
		\begin{equation*}
                L(\sigma, \pi' \times \widetilde{\pi}') \ll (\sigma-1)^{-1} \exp\Big( \Cr{Li} {m'}^2\frac{\log \mathfrak{C}(\pi' \times \widetilde{\pi}')}{\log\log \mathfrak{C}(\pi' \times \widetilde{\pi}')}\Big) \ll (\sigma-1)^{-1}\mathfrak{C}(\pi')^\varepsilon.
		\end{equation*}
    \end{thm}

    The second result is a Brun--Titchmarsh type upper bound with mild dependence on $\pi'$.
    
    \begin{lem}\label{lem:pnt}
    Let $\pi \in \mathfrak{F}_m$ and $\pi' \in \mathfrak{F}_{m'}$. If $x\geq 2$ and $\pi$ satisfies GRC, then
        \[
        \sum_{n \le x}|\Lambda_{\pi \times \pi'}(n)| \ll x\sqrt{\log\mathfrak{C}(\pi')}.
        \]
    \end{lem}
    \begin{proof}
        By Lemma \ref{lem:lem2.2-ST} and the Cauchy--Schwarz inequality, we have that
        \begin{equation}\label{eq:pnt-1}
            \sum_{n \le x}|\Lambda_{\pi \times \pi'}(n)| \le \Big(\sum_{n \le x}\Lambda_{\pi \times \widetilde{\pi}}(n)\Big)^{1/2}\Big(\sum_{n \le x}\Lambda_{\pi' \times \widetilde{\pi}'}(n)\Big)^{1/2}.
        \end{equation}
        Since $\pi$ satisfies GRC, it follows from the prime number theorem that 
        \begin{equation}\label{eq:pnt-4}
        \Big(\sum_{n \le x}\Lambda_{\pi \times \widetilde{\pi}}(n)\Big)^{1/2} \le  \Big(\sum_{n \le x}m^2\Lambda(n)\Big)^{1/2} \ll \sqrt{x}. 
        \end{equation}
        We now fix a smooth compactly supported function $\phi$ with $\phi^{(j)}(y) \ll_j 1$ such that $\phi$ is supported on $[1/2,5/2]$ and at least $1$ on $[1,2]$. Let $\widehat{\phi}$ denote the Mellin transform of $\phi$, and let $\rho$ range over all zeros of $L(s, \pi' \times \widetilde{\pi}')$. By the Mellin inversion formula and pushing the contour all the way to the left, we have that
    \begin{equation}\label{eq:pnt-2}
            \sum_{x \le n \le 2x} \Lambda_{\pi' \times \widetilde{\pi}'}(n) \leq  \sum_{n=1}^{\infty} \Lambda_{\pi' \times \widetilde{\pi}'}(n)\phi(n/x) = x\widehat{\phi}(1)  - \sum_{\rho} x^{\rho}\widehat{\phi}(\rho).
        \end{equation}
          By \eqref{eqn:LRS_finite} and \eqref{eqn:LRS_2},  the contribution from trivial zeros is $\ll x$.  For nontrivial zeros, we have that $|x^{\rho}|\leq x$, trivially. Write $\rho = \beta + i \gamma$. By integration by parts and the derivative estimates of $\phi$, it follows that $|\widehat{\phi}(\rho)| \ll_\phi (|\gamma|^2+1)^{-1}$. Since there are $O(\log \mathfrak{C}(\pi' \times \widetilde{\pi}', t))$ zeros $\rho$ in a unit strip around $t$, it follows that
        \begin{equation}\label{eq:sum-rho}
            \sum_{\rho} \widehat{\phi}(\rho) \ll \int_{-\infty}^{\infty} (\log \mathfrak{C}(\pi'\times \widetilde{\pi}') +\log(|t|+3)) \frac{dt}{|t|^2+1}  \ll \log\mathfrak{C}(\pi').
        \end{equation} 
        By \eqref{eq:pnt-2} and \eqref{eq:sum-rho},  if $x\geq 2$, then $\sum_{x \le n \le 2x} \Lambda_{\pi' \times \widetilde{\pi}'}(n) \ll x\log\mathfrak{C}(\pi')$. The lemma now follows upon using dyadic decomposition, \eqref{eq:pnt-1}, and \eqref{eq:pnt-4}.
    \end{proof}
    
    For future convenience, we record some helpful bounds obtained from applying Lemma \ref{lem:lem3.1-JLW}, the Cauchy--Schwarz inequality, and Theorem \ref{thm:Xiannan-Li}.

    \begin{lem}\label{lem:3.2-Sound}  Let $\pi \in \mathfrak{F}_m$ and $ \pi' \in \mathfrak{F}_{m'}$. Let $\pi$ satisfy GRC. If $x \ge 2$ and $1< \sigma \le 2$, then
		\begin{align}
            &\sum_{n \le x} \frac{|\lambda_{\pi \times \pi'}(n)|}{n}  \ll (\log x)^{\frac{m^2+1}{2}}\mathfrak{C}(\pi')^{\varepsilon},\label{eq:lem3.2-Sound-1}\\
            &|L(\sigma+it, \pi \times \pi')| \ll (\sigma-1)^{-\frac{m^2+1}{2}}\mathfrak{C}(\pi')^{\varepsilon},\label{eq:lem3.2-Sound-2} \\
            &|L'(\sigma+it, \pi \times \pi')| \ll (\sigma-1)^{-\frac{m^2+3}{2}}\mathfrak{C}(\pi')^{\varepsilon}.\label{eq:lem3.2-Sound-3} 
            \end{align}
	\end{lem}
	\begin{proof} 
		We first prove \eqref{eq:lem3.2-Sound-2}.
        By Lemma \ref{lem:lem3.1-JLW} and the Cauchy--Schwarz inequality, we have
        \begin{equation*}
	\begin{aligned}
            |L(\sigma+it, \pi \times \pi')| \le \sum_{n=1}^{\infty} \frac{\sqrt{\lambda_{\pi \times \widetilde{\pi}}(n)\lambda_{\pi' \times \widetilde{\pi}'}(n)}}{n^{\sigma}}
            \le \Big(\sum_{n=1}^{\infty} \frac{\lambda_{\pi \times \widetilde{\pi}}(n)}{n^{\sigma}}\Big)^{1/2}\Big(\sum_{n=1}^{\infty} \frac{\lambda_{\pi' \times \widetilde{\pi}'}(n)}{n^{\sigma}}\Big)^{1/2}.
	\end{aligned}
        \end{equation*}
	 Let $\tau_d(n)$ be the number of ways to write $n$ as a product of $d$ positive integers where the order matters.  Since $\pi \in \mathfrak{F}_m$ satisfies GRC, we have $\lambda_{\pi \times \widetilde{\pi}}(n) \le \tau_{m^2}(n)$.  Then
    \begin{equation}
        \label{eq:lem3.2-Sound-5}
        \Big(\sum_{n=1}^{\infty} \frac{\lambda_{\pi \times \widetilde{\pi}}(n)}{n^{\sigma}}\Big)^{1/2} \le \Big(\sum_{n=1}^{\infty} \frac{\tau_{m^2}(n)}{n^{\sigma}}\Big)^{1/2} = \zeta(\sigma)^{\frac{m^2}{2}} \ll (\sigma-1)^{-\frac{m^2}{2}}.
    \end{equation}
    By Theorem \ref{thm:Xiannan-Li}, we have that 
    \begin{equation}
        \label{eq:lem3.2-Sound-6}
        \Big(\sum_{n=1}^{\infty} \frac{\lambda_{\pi' \times \widetilde{\pi}'}(n)}{n^{\sigma}}\Big)^{1/2} = L(\sigma, \pi' \times \widetilde{\pi}')^{1/2} \ll (\sigma-1)^{-1/2}\mathfrak{C}(\pi')^{\varepsilon}.
    \end{equation} 
    The assertion \eqref{eq:lem3.2-Sound-2} follows from \eqref{eq:lem3.2-Sound-5} and \eqref{eq:lem3.2-Sound-6}. To prove \eqref{eq:lem3.2-Sound-3}, we combine \eqref{eq:lem3.2-Sound-2}  with Cauchy's integral formula for $L'(\sigma +it, \pi \times \pi')$. Lastly, to prove \eqref{eq:lem3.2-Sound-1}, we observe that
    \begin{equation*}
        \label{eq:prelim(1)-1}
        \sum_{n \le x} \frac{|\lambda_{\pi \times \pi'}(n)|}{n} \leq e \sum_{n=1}^{\infty}\frac{|\lambda_{\pi \times \pi'}(n)|}{n^{1+1/\log(ex)}}.
    \end{equation*}
    The result now follows from \eqref{eq:lem3.2-Sound-2}, \eqref{eq:lem3.2-Sound-5}, and \eqref{eq:lem3.2-Sound-6} (each with $\sigma=1+1/\log(ex)$).
    \end{proof}

    \begin{lem}
        \label{lem:3.3-Sound} 
        Let $\pi \in \mathfrak{F}(m)$ and $ \pi' \in \mathfrak{F}(m')$. Let $\pi$ satisfy GRC. If $x \ge 2$, then
	\begin{equation}
        \label{eq:partialsum-rankin-selberg}
            \sum_{n\le x}|\lambda_{\pi\times\pi'}(n)| \ll x(\log x)^{\frac{m^2+1}{2}}\mathfrak{C}(\pi')^{\varepsilon}.
	\end{equation}
        Moreover, if $1 \le y \le x$, then
        \begin{equation}
        \label{eq:shortsum-rankin-selberg}
            \sum_{x < n\le x+y}|\lambda_{\pi\times\pi'}(n)| \ll x^{3/4}y^{1/4}(\log x)^{\frac{m^4+2}{4}}\mathfrak{C}(\pi')^{\varepsilon}.
	\end{equation}	
    \end{lem}

    \begin{proof}
        The first assertion follows from \eqref{eq:lem3.2-Sound-1} of Lemma \ref{lem:3.2-Sound} and Rankin's trick:
	\[
            \sum_{n\le x} |\lambda_{\pi\times\pi'}(n)| \leq x\sum_{n\le x} \frac{|\lambda_{\pi\times\pi'}(n)|}{n} \ll x(\log x)^{\frac{m^2+1}{2}}\mathfrak{C}(\pi')^{\varepsilon}.
	\]
        For the second assertion, it follows from Lemma \ref{lem:lem3.1-JLW} and the Cauchy--Schwarz inequality that 
	\begin{equation*}
            \sum_{x < n \le x+y} |\lambda_{\pi \times \pi'}(n)| \le \Big(\sum_{x < n \le x+y} \lambda_{\pi \times \widetilde{\pi}}(n) \Big)^{1/2}\Big(\sum_{x < n \le x+y}\lambda_{\pi' \times \widetilde{\pi}'}(n)\Big)^{1/2}. 
        \end{equation*}
        If $\pi$ satisfies GRC, then $\lambda_{\pi \times \widetilde{\pi}}(n) \le \tau_{m^2}(n)$. Since $\lambda_{\pi \times \widetilde{\pi}}(n)$ is nonnegative, the Cauchy--Schwarz inequality and Rankin's trick implies that
	\begin{equation*}
            \Big(\sum_{x < n \le x+y} \lambda_{\pi \times \widetilde{\pi}}(n) \Big)^{1/2} \le y^{1/4}\Big(\sum_{x < n \le x+y} \lambda_{\pi \times \widetilde{\pi}}(n)^2 \Big)^{1/4} \le (xy)^{1/4}\Big(\sum_{x < n \le x+y} \frac{\tau_{m^2}(n)^2}{n} \Big)^{1/4}.
	\end{equation*}
        Since $\tau_{m^2}(n)^2 \le \tau_{m^4}(n)$, the above display is bounded above by
        \begin{equation*}
		\begin{aligned}
     (xy)^{1/4}\Big(\sum_{x < n \le x+y} \frac{\tau_{m^4}(n)}{n} \Big)^{1/4} & \leq e (xy)^{1/4}\Big(\sum_{n=1}^{\infty} \frac{\tau_{m^4}(n)}{n^{1+1/\log(2ex)}} \Big)^{1/4} \ll (xy)^{1/4}(\log x)^{{m^4}/{4}}.
		\end{aligned}
	\end{equation*}	
        Since $\lambda_{\pi' \times \widetilde{\pi}'}(n)$ is nonnegative, we apply Rankin's trick  and Theorem \ref{thm:Xiannan-Li} to obtain
	\begin{equation*}
            \Big(\sum_{x < n \le x+y} \lambda_{\pi' \times \widetilde{\pi}'}(n)\Big)^{1/2}
            \leq x^{1/2} \Big(e\sum_{ n =1}^{\infty} \frac{\lambda_{\pi' \times \widetilde{\pi}'}(n)}{n^{1+1/\log(2ex)}}\Big)^{1/2} 
            \ll x^{1/2}(\log x)^{1/2} \mathfrak{C}(\pi')^{\varepsilon}.
		\end{equation*}
  The desired result follows.\end{proof}
        
    
	\section{Weak subconvexity bound}\label{sec:weak-subconvexity-bound}

    We now establish Theorem \ref{thm:weaksub}. Our proof follows the framework introduced by Soundararajan in \cite{S}. Consequently, we will exclude specific details mirroring Soundararajan's method and instead focus on highlighting the distinctive refinements in our approach. 
	\subsection{Notation}\label{subsec: notation} 
    Consider $\pi \in \mathfrak{F}_m$ and $\pi'\in \mathfrak{F}_{m'}$. We assume that $L(s, \pi \times \pi')$ is entire, and $\pi$ satisfies GRC.  Following is a list of the conventions and notation used in this section.
    \begin{itemize}
        \item  $C = \mathfrak{C}(\pi\times \pi')$.
        \item $x\geq 2$, $X\geq\max(x,10)$, and $\varepsilon \in (0,1]$.
        \item $S(x) = \sum_{n\le x} \lambda_{\pi \times \pi'}(n)$ and $\widetilde{S}(x) = \sum_{n \le x} \lambda_{\pi \times \pi'}(n) \log n$.
        \item $R = \lfloor 72m^2/\varepsilon^2\rfloor + 1$, $L = \lfloor 5(m^2+1)R \rfloor$ and $\underline{L} = (L, \ldots, L)$.
        \item $w\in [2,\exp((\log X)^{1/(3R)})]$ is an integer.
        \item $T = \exp((\log\log X)^2)$ and $\underline{\tau}=(\tau_1, \tau_2, \ldots, \tau_R)\in\mathbb{R}^R$ with $|\tau_j| \le T$.
        \item  $l_1, \ldots, l_R, j_1, \ldots, j_R \in\mathbb{Z}$ are nonnegative, $\underline{l} =(l_1,\ldots, l_R)$, and $\underline{j} = (j_1, \ldots, j_R)$.
        \item  $\underline{j} \le \underline{l}$ indicates that $0 \le j_1 \le l_1, \ldots, 0\le j_R \le l_R$.
        \item $\binom{\underline{l}}{\underline{j}} = \binom{l_1}{j_1} \cdots \binom{l_R}{j_R}$, where $\binom{a}{b}$ are the binomial coefficients.
        \item $\mathcal{O}_{\underline{l}}(x,w)= \mathcal{O}_{\underline{l}}(x,w; \underline{\tau}) = \sum_{\underline{j}\le \underline{l}} (-1)^{j_1+\cdots+j_R} \binom{\underline{l}}{\underline{j}} w^{j_1(1+i\tau_1) + \cdots + j_R(1+i\tau_R)} S(x/w^{j_1+\cdots+j_R})$.
        \item $\widetilde{\mathcal{O}}_{\underline{l}}(x,w)= \mathcal{O}_{\underline{l}}(x,w; \underline{\tau}) = \sum_{\underline{j}\le \underline{l}} (-1)^{j_1+\cdots+j_R} \binom{\underline{l}}{\underline{j}} w^{j_1(1+i\tau_1) + \cdots + j_R(1+i\tau_R)} \widetilde{S}(x/w^{j_1+\cdots+j_R})$.
        \item $\mathcal{L}_{\underline{l}}(s) = L(s, \pi\times\pi') \prod_{j=1}^{R}(1-w^{1+i\tau_j -s} )^{l_j}$. 
    \end{itemize}


    \subsection{Main results}
    \label{subsec:main-weak-sub}

    As in Soundararajan's work \cite{S}, the essence of our weak subconvexity proof is the slow variation of the mean values of multiplicative functions. Consider the multiplicative function $\lambda_{\pi \times \pi'}(n)$. We will use the convexity bound for $L(s, \pi \times \pi')$ (see Lemma \ref{lem:Sound-lem4.2}) to prove that there exists a constant $B_{0}=B_0(m,m') >0$ such that
	\[
            \sum_{n\le x} \lambda_{\pi \times \pi'}(n) \ll \frac{x}{\log x} \mathfrak{C}(\pi')^\varepsilon, \qquad x \ge C^{1/2}(\log C)^{B_{0}}. 
	\]
	Next, we prove that for all $B>0$, similar cancellation holds even when $x \ge C^{1/2}(\log C)^{-B}$. In particular, we have the following theorem.

	\begin{thm}\label{thm:partial-sum-thm2-Sound}
		Let $\pi$ and $\pi'$ be as in Theorem \ref{thm:weaksub}. If $\varepsilon>0$, $B>0$, and $x\ge C^{1/2}(\log C)^{-B}$, then
		\begin{equation*}
			\sum_{n\le x} \lambda_{\pi \times \pi'}(n) \ll_{B} \frac{x}{(\log x)^{1-\varepsilon}}\mathfrak{C}(\pi')^\varepsilon.
		\end{equation*}
	\end{thm}

    Recall the notation in Section \ref{subsec: notation}. Theorem \ref{thm:partial-sum-thm2-Sound} follows from our next theorem.
	
	\begin{thm}
        \label{thm:2.1-Sound} With the notation in Section \ref{subsec: notation}, there exist real numbers $\tau_1, \ldots, \tau_R$, whose absolute values are at most $\exp((\log\log X)^2)$, such that if $2 \le x \le X$, then
	\begin{equation*}
            |\mathcal{O}_{\underline{L}}(x,w;\tau_1, \ldots, \tau_R)| \ll \frac{x}{\log x}(\log X)^{\varepsilon} \mathfrak{C}(\pi')^\varepsilon.
	\end{equation*} 
	\end{thm}

    
    \subsection{Deduction of the weak subconvexity bound}
    
    To prove Theorems \ref{thm:partial-sum-thm2-Sound} and \ref{thm:weaksub}, we establish two useful lemmas; the first one is derived from Soundararajan and Thorner in \cite{ST}.
	\begin{lem}\label{lem:Sound-Thorner-lem4.1} Let $\pi \in \mathfrak{F}_m$ and $ \pi' \in \mathfrak{F}_{m'}$, and let $\pi$ satisfy GRC. We have that
		\begin{equation*}
			L(1/2+it, \pi \times \pi') \ll \mathfrak{C}(\pi \times \pi')^{1/4}(|t|+1)^{{mm'}/{4}}.
		\end{equation*}
	\end{lem}
	
	\begin{proof} The lemma follows from Corollary 1.3 of \cite{ST} and the fact that $\pi$ satisfies GRC.
        \end{proof}
	
	As discussed in Section \ref{subsec:main-weak-sub}, the next lemma gives an upper bound for the partial summation $\sum_{n\le x} \lambda_{\pi \times \pi'}(n)$ in the standard range of $x$ upon using the convexity bound.  
	
	\begin{lem}\label{lem:Sound-lem4.2} 
            Let $\pi \in \mathfrak{F}_{m}$ and $\pi' \in \mathfrak{F}_{m'}$. Assume that $L(s,\pi\times\pi')$ is entire and $\pi$ satisfies GRC. If $C = \mathfrak{C}(\pi\times \pi')$, in the range $x \ge x_0 := C^{1/2}(\log C)^{36 m^5m'}$, we have 
		\[
                \sum_{n\le x} \lambda_{\pi\times \pi'} (n) \ll \frac{x}{\log x}\mathfrak{C}(\pi')^\varepsilon.
		\]
	\end{lem}
	\begin{proof}
    Using the same smoothing kernel as in the proof of \cite[Lemma 4.2]{S}, we have that, for any $c >1$, $x > 0$, $\lambda > 0$, and any natural number $K$, \begin{equation}\label{eq:lem4.2-sound-1}
	\frac{1}{2\pi i}\int_{c-i\infty}^{c+i\infty} L(s,\pi\times\pi')\frac{x^s}{s}\Big(\frac{e^{\lambda s}-1}{\lambda s}\Big)^K \,ds = \sum_{n\le x}\lambda_{\pi\times\pi'}(n) + O\Big( \sum_{x < n \le e^{K\lambda}x} |\lambda_{\pi\times\pi'}(n)| \Big).
	\end{equation}
    We then proceed as in the proof of \cite[Lemma 4.2]{S} except we take $K =\lfloor mm'/4\rfloor+2$ and $\lambda = (\log x)^{-m^4-6}$ and use Lemmas \ref{lem:3.3-Sound} and \ref{lem:Sound-Thorner-lem4.1} instead of \cite[Lemmas 3.3, 4.1]{S}, respectively.
        \end{proof}

    \begin{proof}[Proof of Theorem \ref{thm:partial-sum-thm2-Sound}, assuming Theorem 4.2] 

    Replacing \cite[Theorem 2.1]{S} with Theorem \ref{thm:2.1-Sound} and \cite[Lemma 4.2]{S} with Lemma \ref{lem:Sound-lem4.2}, we proceed as in the proof of \cite[Theorem 2]{S}.
	\end{proof}
	
	\begin{proof}[Proof of Theorem \ref{thm:weaksub}]
	We use an approximation functional equation for $L(1/2, \pi \times \pi')$. Following the proof of \cite[Theorem 1]{S},  we have that, for $c > 1/2$
            \begin{multline}
            \label{eqn:AFE}
			L(\tfrac{1}{2}, \pi\times \pi') = \frac{1}{2\pi i}\int_{c-i\infty}^{c+i\infty} L(s+\tfrac{1}{2},\pi\times \pi') \frac{L_{\infty}(s+\tfrac{1}{2},\pi\times \pi')}{L_{\infty}(\tfrac{1}{2}, \pi\times \pi')}e^{s^2} \frac{ds}{s} \\
			+ \frac{W(\pi\times\pi')}{2\pi i}\int_{c-i\infty}^{c+i\infty} L(s+\tfrac{1}{2},\widetilde{\pi}\times \widetilde{\pi}') \frac{L_{\infty}(s+\tfrac{1}{2},\widetilde{\pi}\times \widetilde{\pi}')}{L_{\infty}(\tfrac{1}{2}, \widetilde{\pi}\times \widetilde{\pi}')}e^{s^2} \frac{ds}{s}.
            \end{multline}
	Both integrals above can be similarly estimated, so we will only deal with the first one. Using partial summation for $L(s+1/2,\pi \times \pi')$, we have that the first integral in \eqref{eqn:AFE} equals
	\begin{equation}
        \label{eq:thm1-sound-1}
			\int_1^\infty \sum_{n\le x} \lambda_{\pi\times\pi'}(n) \Big( \frac{1}{2\pi i}\int_{c-i\infty}^{c+i\infty} (s+\tfrac{1}{2}) \frac{L_{\infty}(s+\tfrac{1}{2},\pi\times \pi')}{L_{\infty}(\tfrac{1}{2}, \pi\times \pi')}e^{s^2}x^{-s} \frac{ds}{s} \Big) \frac{dx}{x^{3/2}}.
	\end{equation}
		
		To estimate the inner integral of \eqref{eq:thm1-sound-1}, we move the line of integration either to $\Real(s) =2$, or $\Real(s) = (\theta_m+\theta_{m'})/{2}$. By Stirling's formula, the inner integral of \eqref{eq:thm1-sound-1} is 
		\[
		      \ll \min\big\{(\sqrt{C}/x)^{{(\theta_m+\theta_{m'})}/{2}}, (\sqrt{C}/x)^2\big\}.
		\] 
		Thus, \eqref{eq:thm1-sound-1} is
    \begin{multline*}
        \ll C^{(\theta_m+\theta_{m'})/{4}}\int_1^{\sqrt{C}/(\log C)^{(m^2+3)/(1-\theta_m-\theta_{m'})}}\Big|\sum_{n\le x}\lambda_{\pi\times \pi'}(n)\Big| \frac{dx}{x^{(3+\theta_m+\theta_{m'})/{2}}}\\
        +C^{(\theta_m+\theta_{m'})/{4}}\int_{\sqrt{C}/(\log C)^{(m^2+3)/(1-\theta_m-\theta_{m'})}}^{\sqrt{C}}\Big|\sum_{n\le x}\lambda_{\pi\times \pi'}(n)\Big| \frac{dx}{x^{(3+\theta_m+\theta_{m'})/{2}}}\\
        + C \int_{\sqrt{C}}^{\infty} \Big|\sum_{n\le x}\lambda_{\pi\times \pi'}(n)\Big|\frac{dx}{x^{7/2}}.
    \end{multline*}
    The first integral is $\ll \mathfrak{C}(\pi')^\varepsilon C^{1/4}/\log C$ by \eqref{eq:partialsum-rankin-selberg} of Lemma \ref{lem:3.3-Sound}.  The second and third integrals are $\ll \mathfrak{C}(\pi')^\varepsilon C^{1/4}/(\log C)^{1-\varepsilon}$ by Theorem  \ref{thm:partial-sum-thm2-Sound}. The desired result follows.
	\end{proof}

    
	\subsection{Successive maxima} 
	 
    We now define $\tau_1, \ldots, \tau_R$ appearing in Theorem \ref{thm:2.1-Sound}.
    Define $\tau_1$ as the point $t$ in the compact set $\mathcal{S}_1 = [-T, T]$ where the maximum of $|L(1+1/\log X +it, \pi \times \pi')|$ is attained. Removing the interval $(\tau_1 - (\log X)^{-1/R}, \tau_1 + (\log X)^{-1/R})$ from $\mathcal{S}_1$, we denote by $\mathcal{S}_2$ the remaining compact set. We then define $\tau_2$ to be the point in $\mathcal{S}_2$ where the maximum $|L(1+1/\log X +it, \pi \times \pi')|$ is attained. Continuing the process, we obtain the nested compact sets $\mathcal{S}_1 \supset \mathcal{S}_2 \supset \ldots \supset \mathcal{S}_R$. We observe that if $j \neq k$, then $|\tau_j - \tau_k| \ge (\log X)^{-1/R}$. By this construction, we obtain the following improved version of \eqref{eq:lem3.2-Sound-2} of Lemma \ref{lem:3.2-Sound}.
	\begin{lem}\label{lem:sound-5.1}
   Let  $1 \le j \le R$ and $t$ be a point in $\mathcal{S}_j$.  If $\varepsilon > 0$, then \begin{equation*}    |L(1+1/{\log X}+it, \pi \times \pi')| \ll \mathfrak{C}(\pi')^\varepsilon (\log X)^{\frac{3m}{2}\sqrt{\frac{1}{j}+\frac{j-1}{jR}}+\frac{\varepsilon}{4}}.
    \end{equation*}
		In particular, if $t \in \mathcal{S}_R$, we have $|L(1+1/\log X+it, \pi \times \pi')| \ll \mathfrak{C}(\pi')^{\varepsilon}(\log X)^{\varepsilon/2}$.
	\end{lem}
	\begin{proof}
		Without any modification, \cite[Lemma 5.1]{S} gives 
		\begin{equation}\label{eq:lem5.1-sound}
			|L(1+1/\log X+i\tau_j, \pi \times \pi')| 
			\le \exp\Big(\Real \Big(\frac{1}{j}\sum_{n \ge 2} \frac{\Lambda_{\pi \times \pi'}(n)}{n^{1+1/\log X}\log n}\sum_{r=1}^{j}n^{-i\tau_r} \Big)\Big).
		\end{equation}
		By Lemma \ref{lem:lem2.2-ST} and the Cauchy--Schwarz inequality, we have that
	\begin{equation*}
	\begin{aligned}
		\sum_{n \ge 2} &\frac{|\Lambda_{\pi \times \pi'}(n)|}{n^{1+1/\log X}\log n}\Big|\sum_{r=1}^{j} n^{-i\tau_r}\Big| 
        \le \Big(\sum_{n \ge 2} \frac{\Lambda_{\pi \times \widetilde{\pi}}(n)}{n^{1+1/\log X}\log n}\Big|\sum_{r=1}^{j} n^{-i\tau_r}\Big|^2\Big)^{1/2} \Big(\sum_{n \ge 2} \frac{\Lambda_{\pi' \times \widetilde{\pi}'}(n)}{n^{1+1/\log X}\log n}\Big)^{1/2}.
	\end{aligned}
    \end{equation*}
     Applying Theorem \ref{thm:Xiannan-Li}, the second factor above equals
        \begin{equation}
        \label{eq:lem5.1-sound-3} 
	\begin{aligned}
            \big(\log L(1+1/{\log X}, \pi' \times \widetilde{\pi}')\big)^{1/2} &\le \Big(\log\log X + c_1m'^2 \frac{\log\mathfrak{C}(\pi' \times \widetilde{\pi}')}{\log\log\mathfrak{C}(\pi' \times \widetilde{\pi}')}+O(1)\Big)^{1/2}.
		\end{aligned}
	\end{equation} 
     For the first factor, since $\pi$ satisfies GRC and  $\Lambda_{\pi \times \widetilde{\pi}}(n)$ is supported on prime powers, we obtain the bound $\Lambda_{\pi \times \widetilde{\pi}}(n) \le m^2\Lambda(n)$. By the calculations on \cite[p. 1488]{S}, it follows that  \begin{equation}\label{eq:lem5.1-sound-2}
       \Big(\sum_{n \ge 2} \frac{\Lambda_{\pi \times \widetilde{\pi}}(n)}{n^{1+1/\log X}\log n}\Big|\sum_{r=1}^{j} n^{-i\tau_r}\Big|^2\Big)^{1/2}
        \le m\Big(\Big(j +\frac{j(j-1)}{R}\Big)\log\log X + O(j) \Big)^{1/2}.
	\end{equation}
    
    Using that $\sqrt{a+b+c}\leq \sqrt{a}+\sqrt{b}+\sqrt{c}$ for $a,b,c \geq 0$,  \eqref{eq:lem5.1-sound-3}, and \eqref{eq:lem5.1-sound-2}, we bound \eqref{eq:lem5.1-sound} by
	\begin{equation}\label{eq:lem5.1-sound-7}
        \begin{aligned}
        \exp &\Bigg[m \sqrt{\frac{1}{j}+\frac{j-1}{jR}}\Bigg(\log\log X + 
        \sqrt{\log\log X}\sqrt{c_1 {m'}^2\frac{\log\mathfrak{C}(\pi' \times \widetilde{\pi}')}{\log\log\mathfrak{C}(\pi' \times \widetilde{\pi}')}}\Bigg) \\ 
        & +O\Bigg( \Bigg(m\sqrt{\frac{1}{j}+\frac{j-1}{jR}}+ \frac{m}{j}\Bigg)\sqrt{\log\log X} + \frac{mm'}{\sqrt{j}}\sqrt{c_1 \frac{\log\mathfrak{C}(\pi' \times \widetilde{\pi}')}{\log\log\mathfrak{C}(\pi' \times \widetilde{\pi}')}}\Bigg)\Bigg].
        \end{aligned}
    \end{equation} 
    By the inequality of arithmetic and geometric means, it follows that
    \begin{equation}
    \label{eq:lem5.1-sound-8}
    \sqrt{\log\log X}\sqrt{c_1 m'^2 \frac{\log\mathfrak{C}(\pi' \times \widetilde{\pi}')}{\log\log\mathfrak{C}(\pi' \times \widetilde{\pi}')}} \le \frac{1}{2}\log\log X + \frac{c_1m'^2}{2} \frac{\log\mathfrak{C}(\pi' \times \widetilde{\pi}')}{\log\log\mathfrak{C}(\pi' \times \widetilde{\pi}')}.
    \end{equation}
    Hence, the lemma follows from \eqref{eq:lem5.1-sound-7} and \eqref{eq:lem5.1-sound-8}.
    \end{proof}

    In the proof of Theorem \ref{thm:2.1-Sound}, there will appear the quantity 
    $\mathcal{L}_{\underline{l}}(s)$, defined in Section \ref{subsec: notation}. With the notation in Section \ref{subsec: notation}, we record  useful bounds for $\mathcal{L}_{\underline{l}}(s)$ in the next two lemmas. 
    \begin{lem}\label{lem:sound-5.2}
		 For any $\sigma \ge 1+1/\log X$, we have that
        \begin{equation*}
			\max_{|t| \le T/2} |\mathcal{L}_{\underline{l}}(\sigma+it)| \le \max_{|t| \le T}|\mathcal{L}_{\underline{l}}(1+1/\log{X}+it)| + O((\log X)^{-1}\mathfrak{C}(\pi')^\varepsilon).
		\end{equation*}
    \end{lem}
    \begin{proof}
		We follow the proof of \cite[Lemma 5.2]{S}, but replace \cite[Lemma 3.2]{S} with \eqref{eq:lem3.2-Sound-2} 
	\end{proof}
	\begin{lem}\label{lem:sound-5.3}
     Let $\underline{L-1}=(L-1,\ldots,L-1)$.  If $\underline{l}\geq \underline{L-1}$, then
		 \begin{equation*}
			\max_{|t| \le T} |\mathcal{L}_{\underline{l}}(1+1/\log{X}+it )| \ll (\log X)^{\varepsilon/2}\mathfrak{C}(\pi')^{\varepsilon/2}.
		\end{equation*}
	\end{lem}
	\begin{proof}
		Replacing \cite[Lemma 3.2]{S}  with \eqref{eq:lem3.2-Sound-2} and \cite[Lemma 5.1]{S} with Lemma \ref{lem:sound-5.1}, we proceed just as in  the proof of \cite[Lemma 5.3]{S}.
	\end{proof}

	\begin{prop}\label{prop:sound-5.4}
	    Recall the notation and hypotheses of Lemma \ref{lem:sound-5.3}.   If $0 \le \log w \le (\log X)^{1/(3R)}$ and $1 \le x \le X$, then
		\begin{equation*}
			\mathcal{O}_{\underline{l}}(x, w) \ll x(\log X)^{2\varepsilon/3}\mathfrak{C}(\pi')^\varepsilon.
		\end{equation*}
	\end{prop}
	\begin{proof}
		By \eqref{eq:partialsum-rankin-selberg}, we may assume that $\log x \ge (\log X)^{\varepsilon/(m^2+1)}$. If $w \le \exp((2RL)^{(m^2+1)/(3R\varepsilon-m^2-1)})$, then the desired bound is trivial. Thus, we may assume otherwise,  which implies that $x \ge w^{2RL}$.  
        By  \eqref{eq:lem4.2-sound-1} with $K=1$, $\lambda=1/\sqrt{T}$, and $c>1$, we find that 
		\begin{equation*}
			\frac{1}{2\pi i}\int_{c-i\infty}^{c+i\infty} L(s, \pi \times \pi') z^s\Big(\frac{e^{s/\sqrt{T}-1}}{s/\sqrt{T}}\Big)\frac{ds}{s} = \sum_{n\le z} \lambda_{\pi \times \pi'}(n) + O\Big(\sum_{z<n\le ze^{1/\sqrt{T}}}|\lambda_{\pi\times\pi'}(n)|\Big).
		\end{equation*}
        The rest of the proof follows the proof of \cite[Proposition 5.4]{S}, but we replace \cite[Lemma 3.2]{S} with \eqref{eq:lem3.2-Sound-2}, \cite[Lemma 3.3]{S} with \eqref{eq:shortsum-rankin-selberg}, and \cite[Lemma 5.3]{S} with Lemma \ref{lem:sound-5.3}.
        \end{proof}

 
    \subsection{Proof of Theorem \ref{thm:2.1-Sound}}
	In this section, we recall the notation  in Sections \ref{sec:L-functions} and \ref{subsec: notation}.
	
	\begin{lem}\label{lem:6.1-Sound} 
		 If $0 \le \log w \le (\log X)^{1/(3R)}$ and $1 \le x \le X$, then 
		\begin{equation*}
			(\log x)\mathcal{O}_{\underline{L}}(x, w) = \sum_{d \le x} \Lambda_{\pi \times \pi'}(d)\mathcal{O}_{\underline{L}}(x/d, w) + O(x(\log X)^\varepsilon \mathfrak{C}(\pi')^\varepsilon).
		\end{equation*}
	\end{lem}
	\begin{proof}
		To prove this lemma, we follow the proof of \cite[Lemma 6.1]{S} but replace \cite[Proposition 5.4]{S} with Proposition \ref{prop:sound-5.4}.
        \end{proof}
		
	\begin{lem}
        \label{lem:6.2-Sound} 
         If $1 \le z \le y$ and $y+z \le X$, then 
	\begin{equation*}		
            ||\mathcal{O}_{\underline{L}}(y, w)|^2- |\mathcal{O}_{\underline{L}}(y+z, w)|^2| \ll \mathfrak{C}(\pi')^\varepsilon y(\log X)^\varepsilon \sum_{j=0}^{LR} w^j \sum_{y/w^j<n\le (y+z)/w^j}|\lambda_{\pi\times\pi'}(n)|.
	\end{equation*}
	\end{lem}
	\begin{proof}
        Replacing \cite[Proposition 5.4]{S} with Proposition \ref{prop:sound-5.4} and using the definition of $S(x)$, we proceed just as in the proof of \cite[Lemma 6.2]{S}.
        \end{proof}

	\begin{prop}\label{prop:6.3-Sound}  If $0 \le \log w \le (\log X)^{1/(3R)}$ and $1 \le x \le X$, then 
		\begin{multline*}
	|\mathcal{O}_{\underline{L}}(x, w)|\log x\\ 
		 \ll x (\log\log x)^{1/2}\mathfrak{C}(\pi')^{{\varepsilon}/{4}}\Big(\int_1^x \log(ey)|\mathcal{O}_{\underline{L}}(y,w)|^2 \, \frac{dy}{y^3}+\mathfrak{C}(\pi')^\varepsilon\Big)^{1/2} + x(\log X)^{\varepsilon}\mathfrak{C}(\pi')^\varepsilon.
		\end{multline*}
	\end{prop}
	
	\begin{proof}
		By Lemma \ref{lem:6.1-Sound}, it remains to estimate $\sum_{d \le x}|\Lambda_{\pi \times \pi'}(d)||\mathcal{O}_{\underline{L}}(x/d,w)|$. We consider this sum in two ranges: $d \le D := \lfloor \exp((\log\log X)^6)\rfloor$ and $d > D$. By Proposition \ref{prop:sound-5.4}, we have
		\begin{equation}\label{eq:lem6.3-sound-0}
			\sum_{d \le D}|\Lambda_{\pi \times \pi'}(d)||\mathcal{O}_{\underline{L}}(x/d,w)| \ll x(\log X)^{2\varepsilon/3}\sum_{d \le D}\frac{|\Lambda_{\pi \times \pi'}(d)|}{d}.
		\end{equation}
		By Lemma \ref{lem:pnt} and partial summation, the right-hand side of \eqref{eq:lem6.3-sound-0} is	\begin{equation}\label{eq:lem6.3-sound-1}
                \ll x(\log X)^{2\varepsilon/3} \log(eD)\mathfrak{C}(\pi')^\varepsilon \ll x(\log X)^{\varepsilon} \mathfrak{C}(\pi')^\varepsilon .
	\end{equation}
 
     Let $g(t) = t \log(ex/t)$ for $1 \le t \le x$. By Lemma \ref{lem:lem3.1-JLW} and the Cauchy--Schwarz inequality, the sum $\sum_{D < d \le x}|\Lambda_{\pi \times \pi'}(d)||\mathcal{O}_{\underline{L}}(x/d,w)|$ is \begin{equation}
    \label{eq:lem6.3-sound-2}
        \le \Big(\sum_{D < d \le x}\Lambda_{\pi \times \widetilde{\pi}}(d)g(d)|\mathcal{O}_{\underline{L}}(x/d,w)|^2\Big)^{1/2} \Big(\sum_{D < d \le x}\frac{\Lambda_{\pi' \times \widetilde{\pi}'}(d)}{g(d)}\Big)^{1/2}.
     \end{equation}
        By using GRC for $\Lambda_{\pi \times \widetilde{\pi}}(d)$, and Lemma \ref{lem:pnt} and partial summation for $\Lambda_{\pi' \times \widetilde{\pi}'}(d)$,  \eqref{eq:lem6.3-sound-2} is 
	\begin{equation*}
            \ll  \Big(\sum_{D < d \le x}g(d)\Lambda(d)|\mathcal{O}_{\underline{L}}(x/d,w)|^2\Big)^{1/2}(\log\log x)^{1/2}\mathfrak{C}(\pi')^{{\varepsilon}/{4}}.
	\end{equation*}
 
		We now follow the proof of \cite[Proposition 6.3]{S} but replace \cite[Lemma 3.2]{S} with \eqref{eq:lem3.2-Sound-1}, \cite[Proposition 5.4]{S} with Proposition \ref{prop:sound-5.4}, and \cite[Lemma 6.2]{S} with Lemma \ref{lem:6.2-Sound}.
        \end{proof}

	\begin{lem}\label{lem:6.4-Sound}
		 If $0 \le \log w \le (\log X)^{1/(3R)}$ and $1 \le x \le X$, then 
		\begin{equation*}
			\Big(\int_1^x |\mathcal{O}_{\underline{L}}(t, w)|^2 \log(et) \frac{dt}{t^3}\Big)^{1/2} \ll \Big(\int_1^x|\widetilde{\mathcal{O}}_{\underline{L}}(t,w)|^{2} \frac{dt}{t^3\log(et)} \Big)^{1/2} + (\log X)^{{7\varepsilon}/{8}} \mathfrak{C}(\pi')^\varepsilon.
		\end{equation*}
	\end{lem}
	\begin{proof}  Replacing \cite[Proposition 5.4]{S} with our Proposition \ref{prop:sound-5.4}, we then proceed as in the proof of \cite[Lemma 6.4]{S}.
    \end{proof}

	Combining Proposition \ref{prop:6.3-Sound} and Lemma \ref{lem:6.4-Sound}, we have that
	\begin{equation}\label{eq:Thm2.1}
		|\mathcal{O}_{\underline{L}}(x,w)|\ll \frac{x(\log\log x)^{1/2}\mathfrak{C}(\pi')^{\varepsilon/4}}{\log x}  \Big(\int_1^x |\widetilde{\mathcal{O}}_{\underline{L}}(t,w)|^2 \frac{dt}{t^3\log(et)} \Big)^{1/2} +\frac{x (\log X)^\varepsilon\mathfrak{C}(\pi')^\varepsilon}{\log x}.
	\end{equation}
	We estimate the integral in \eqref{eq:Thm2.1} as follows. 
	
	\begin{prop}\label{prop:6.5-Sound} 
		If $0 \le \log w \le (\log X)^{1/(3R)}$ and $1 \le x \le X$, then 
		\begin{equation*}
			\int_1^x|\widetilde{\mathcal{O}}_{\underline{L}}(t,w)|^{2} \frac{dt}{t^3\log(et)} \ll (\log X)^{3\varepsilon/2}\mathfrak{C}(\pi')^{3\varepsilon/2}.
		\end{equation*}
	\end{prop}
	
	\begin{proof}
        Without any modification, the proof of \cite[Proposition 6.5]{S} gives  
    \begin{equation}\label{eq:prop6.5-sound-1}
		\begin{aligned}
                &\int_1^x|\widetilde{\mathcal{O}}_{\underline{L}}(t,w)|^{2} \frac{dt}{t^3\log(et)}\\ 
                &\quad \ll \int_{1/\log X}^\infty e^{-2\alpha}\int_{-\infty}^{\infty} |L'(1+\alpha+it, \pi \times \pi')|^2\prod_{k=1}^{R} |1- w^{-\alpha-it+i\tau_k}|^{2L} \frac{dt}{|1+\alpha+it|^2} d\alpha. 
			\end{aligned}
		\end{equation}
  
	We split the inner integral above into two ranges: $|t| \le T/2$ and $|t| > T/2$. 
    By \eqref{eq:lem3.2-Sound-3}, the inner integral in \eqref{eq:prop6.5-sound-1} over the range $|t| > T/2$ contributes 
  \begin{equation}
            \label{eq:prop6.5-sound-geT/2}
		\ll (\log X)^{m^2+3} \mathfrak{C}(\pi')^\varepsilon \int_{|t| > T/2} \frac{dt}{|1+\alpha+it|^2} \ll_A  \frac{\mathfrak{C}(\pi')^\varepsilon}{(\log X)^A},
		\end{equation}
        for any $A > 0$.
        When $|t|\leq T/2$, Lemmas \ref{lem:sound-5.2} and \ref{lem:sound-5.3} yield 
    \begin{equation*}
	\begin{aligned}
        |L'(1+\alpha+it, \pi \times \pi')|\prod_{k=1}^{R} |1- w^{-\alpha-it+i\tau_k}|^{L} 
	&
        \ll \Big|\frac{L'}{L}(1+\alpha+it, \pi \times \pi')\Big| (\log X)^{\varepsilon/2}\mathfrak{C}(\pi')^{\varepsilon/2}.
	\end{aligned}
    \end{equation*}
    Therefore, the inner integral in \eqref{eq:prop6.5-sound-1} over $|t|\leq T/2$ contributes 
		\begin{equation}\label{eq:prop6.5-sound-3}
			\ll (\log X)^{\varepsilon} \mathfrak{C}(\pi')^{\varepsilon}\int_{-\infty}^{\infty} \Big|\frac{L'}{L}(1+\alpha+it, \pi \times \pi')\Big|^2 \frac{dt}{|1+\alpha+it|^2}.
		\end{equation}

		Observe that the Fourier transform of the function $e^{-y(1+\alpha)} \sum_{n\le e^y} \Lambda_{\pi \times \pi'}(n)$ is 
		\begin{equation*}
			\int_{-\infty}^{\infty} \sum_{n \le e^y} \Lambda_{\pi\times\pi'} (n) e^{-y(1+\alpha+it)} dy = \sum_{n=1}^\infty \frac{\Lambda_{\pi\times \pi'}(n)}{n^{1+\alpha+it}}\frac{1}{1+\alpha+it} = -\frac{L'}{L}(1+\alpha+it) \frac{1}{1+\alpha+it}. 
		\end{equation*}
    Thus, by Plancherel's theorem and Lemma \ref{lem:pnt}, we have that \eqref{eq:prop6.5-sound-3} is 
		\begin{equation}
            \label{eq:prop6.5-sound-leT/2}
			\ll (\log X)^\varepsilon \mathfrak{C}(\pi')^{\varepsilon} \int_{-\infty}^{\infty}\Big|\sum_{n\le e^y}\Lambda_{\pi\times \pi'}(n)\Big|^2 e^{-(2+2\alpha)y} dy \ll \frac{(\log X)^\varepsilon \mathfrak{C}(\pi')^{3\varepsilon/2}} {2\alpha},
		\end{equation}
   using $y \ge 0$ in the last step. The proposition follows by using \eqref{eq:prop6.5-sound-geT/2} and \eqref{eq:prop6.5-sound-leT/2}  in \eqref{eq:prop6.5-sound-1}. 
  \end{proof}
        
        \begin{proof}[Proof of Theorem    \ref{thm:2.1-Sound}]
        Theorem \ref{thm:2.1-Sound} follows upon using Proposition \ref{prop:6.5-Sound} in \eqref{eq:Thm2.1}. 
        \end{proof}


	\section{Background on \texorpdfstring{$\GL(2)$}{GL(2)} automorphic forms}\label{sec:GL(2)}
     We recall here some standard properties of
    classical automorphic forms, including holomorphic newforms, Hecke--Maa{\ss} cusp forms, Eisenstein series, incomplete Eisenstein series, and incomplete Poincar\'e series. We will also discuss the spectral decomposition  and the Poincar\'e decomposition. For complete definitions and proofs, see, e.g., \cite{IK}, \cite{NPS}, and \cite{O}. 

	
	\subsection{Holomorphic newforms.}\label{subsec:newforms}
	
	Let $k \geq 2$ be an even integer and $\alpha \in \GL_2 (\mathbb{R})$ with positive determinant, where $\alpha$ acts on $\h$ by fractional linear transformations. Given a function $f: \h \to \mathbb{C}$, we denote $j((\begin{smallmatrix} a & b\\ c & d \end{smallmatrix}),z) = cz+d$ and recall that $f|_k \alpha(z) = (\det\alpha)^{k/2}j(\alpha, z)^{-k}f(\alpha z)$.

    A \textit{holomorphic cusp form} on $\Gamma_0(q)$ of weight $k$ with trivial nebentypus is a holomorphic function $f: \h \to \mathbb{C}$ such that $f|_k \gamma = f$ for all $\gamma \in \Gamma_0(q)$ and $f$ vanishes at the cusps of $\Gamma_0(q)$. 
    We denote the vector space of holomorphic cusp forms on $\Gamma_0(q)$ of weight $k$ with trivial nebentypus as $S_k(\Gamma_0(q))$.
    
    Define $S_k^{\new}(\Gamma_0(q))$ to be the subspace of $S_k(\Gamma_0(q))$ that is orthogonal to the subspace of oldforms, denoted as $S_k^{\textup{old}}(\Gamma_0(q))$, with respect to the Petersson inner product.
    A \textit{holomorphic newform} is a cusp form in $S_k^{\new}(\Gamma_0(q))$ that is also an eigenform of all of the Hecke operators and all of the Atkin--Lehner involutions. The space $S_k^{\new}(\Gamma_0(q))$ has a basis of normalized newforms, which we denote as $\mathcal{B}_k(q)$ (see more discussion in \cite[Chapter 2, Section 5]{O}).

    Consider $f \in \mathcal{B}_k(q)$. Let $e(z) = e^{2\pi iz}$ and let $\lambda_f(n)$ be the eigenvalue of the $n$-th Hecke operator $T_n$. Then the Fourier expansion of $f$ can be written as
	\begin{equation}
        \label{eq:Fourier-newform}
		f(z) = \sum_{n \ge 1} \lambda_f(n)n^{\frac{k-1}{2}}e(nz),
	\end{equation} 
    where $\lambda_f(1) = 1$. If we let $\tau(n)$ denote the number of positive divisors of $n$, then $\lambda_f(n)$ is real, multiplicative and satisfies Deligne's bound $|\lambda_f(n)| \le \tau(n)$.
    
    For any $\gamma \in \Gamma_0(q)$ and $f,g \in \mathcal{B}_k(q)$, if we write $z' = \gamma z = x'+iy'$, then 
        $y'^{k/2}f(z') = (j(\gamma,z)/|j(\gamma,z)|)^ky^{k/2}f(z).$  
    This observation implies the $\Gamma_0(q)$-invariance for both $z \mapsto y^{k}|f(z)|^2$ and $z \mapsto y^{k}f(z)\overline{g(z)}$.
    Recall that we write  $F_k(z) = \rho_f(1)y^{k/2}f(z)$,
    where $\rho_f(1)$ is chosen such that $ \langle 1, |F_k(z)|^2 \rangle_q = 1$. By the Rankin--Selberg theory, we  have \begin{equation}\label{eq:rho_f(1)}
		|\rho_f(1)|^2 
  \asymp \frac{(4\pi)^{k}}{kq\Gamma(k-1)L(1, \ad f)},
	\end{equation}
    where the implied constants are absolute.

	
	\subsection{Cusps of \texorpdfstring{$\Gamma_0(q)$}{G} }\label{subsec:cusps}
	We collect here some properties of cusps of $\Gamma_0(q)$ (see \cite[Section 3.4]{NPS} for more details). Let $\tau$ traverse a set of representatives for the double coset space $\Gamma_{\infty}\bs\Gamma_0(1)/\Gamma_0(q)$. Then $\mathfrak{a}:= \tau^{-1} \infty \in \mathbb{P}^1(\mathbb{Q})$ traverses a set of inequivalent cusps of $\Gamma_0(q)$. Denote the set of cusps of $\Gamma_0(q)$ as $\mathcal{C} = \mathcal{C}(\Gamma_0(q))=\{\mathfrak{a}_j\}_j$. The width of $\mathfrak{a}_j = \tau_j^{-1} \infty$ is given by $w_j = [\Gamma_\infty: \Gamma_\infty \cap \tau_j\Gamma_0(q)\tau_j^{-1}]$, and the scaling matrix of $\mathfrak{a}_j$ is 
\begin{equation}\label{eq:scalingmatrix}
		\sigma_j = \tau_j^{-1} \begin{pmatrix}w_j& \\&1\end{pmatrix}.
	\end{equation}
 
	If the bottom row of $\sigma_j$ is $(c_j, d_j)$, then we  define $c_j$ as the denominator of the cusp $\mathfrak{a}_j$.  Such $c_j$ is a positive divisor of $q$, and we have that the width $w_j$ equals $q/(c_j^2,q)= [q/c_j^2, 1]$. For each positive divisor $c|q$, if we denote the set of cusps of denominator $c$ by 
	\begin{equation}\label{def:C[c]}
	    \mathcal{C}[c] := \{\mathfrak{a}_j \in \mathcal{C}: c_j = c\},
	\end{equation}
    then $\#\mathcal{C}[c]$ equals $\varphi((c,q/c))$. If $q$ is squarefree, then $\#\mathcal{C}[c]=1$. In that case, the multiset of the widths of cusps of $\Gamma_0(q)$ is $\{d \in \mathbb{Z}: d|q, d \ge 1\}$.

    We now collect necessary properties of the Fourier expansion of $|f|^2$ at the cusp $\mathfrak{a}_j$. For any positive divisor $c|q$ and $d \in (\mathbb{Z}/(c,q/c))^{\times}$, we write $\mathfrak{a}_{d/c} \in \mathcal{C}(\Gamma_0(q))$ for the corresponding cusp. For $\mathfrak{a}_j = \mathfrak{a}_{d/c}$, we may take $\tau_j = \begin{psmallmatrix} * & *\\c & d'\end{psmallmatrix}$ for any integer $d'$ for which $(d', c) = 1$ and $d' \equiv d \pmod{(c, q/c)}$. Such $\tau_j$'s then form a set of representatives of $\Gamma_{\infty}\bs\Gamma_0(1)/\Gamma_0(q)$, and the corresponding scaling matrix $\sigma_j$ is defined as in \eqref{eq:scalingmatrix}.
    
    Write $z_j = x_j + iy_j$ for the change of variable $z_j := \sigma_j^{-1}z$.  We observe that $|f(z)|^2y^k = |f(\sigma_jz_j)|^2\Imag(\sigma_jz_j)^k = |f|_k\sigma_j(z_j)|^2 y_j^k$. For some coefficients $\lambda_j(n) \in \mathbb{C}$, we may write
        \begin{equation}\label{eq:fourier-zj}
		f|_k\sigma_j(z_j) = \sum_{n \ge 1} \lambda_j(n)n^{\frac{k-1}{2}}e(nz_j).
	\end{equation}
    In the special case that $\mathfrak{a}_j = \infty$, we observe that $\lambda_j(n) = \lambda(n).$

	In general, $\lambda_j$ is not multiplicative, so we instead work with the root-mean-square of $\lambda_j$ taken over all cusps of a given denominator. For each positive divisor $c$ of $q$, we define
	\begin{equation}\label{eq:lambdac}
		\lambda_{[c]}(n) := \Big(\frac{1}{\#\mathcal{C}[c]}\sum_{\mathfrak{a}_j \in \mathcal{C}[c]} |\lambda_j(n)|^2 \Big)^{1/2}.
	\end{equation}
   With the above definition and the Cauchy--Schwarz inequality, we observe that
   \begin{equation}\label{eq:regrouping}
		\sum_{a_j \in \mathcal{C}} |\lambda_j(n_1)\lambda_j(n_2)|\le \sum_{c\mid q}\#\mathcal{C}[c]\lambda_{[c]}(n_1)\lambda_{[c]}(n_2). 
		\end{equation}

    
	\subsection{Hecke--Maa{\ss} cusp forms}\label{subsec:maasscuspforms}
	A \textit{Maa{\ss} cusp form} $\phi$ of level 1 is an $\Gamma_0(1)$-invariant eigenfunction of the hyperbolic Laplacian $\Delta := -y^{2}(\partial_x^2 + \partial_y^2)$ on $\h$ that decays rapidly at infinity. We denote the spectral parameter of $\phi$ by $t_\phi$, which is given by $\Delta \phi = (1/4+t_\phi^2)\phi$.
    A \textit{Hecke--Maa{\ss} cusp form} is a Maa{\ss} cusp form that is also an eigenfunction of all of the Hecke operators and the involution $T_{-1}: \phi \mapsto [z \mapsto \phi(-\overline{z})],$ which commute with one another as well as with $\Delta$. A Hecke--Maa{\ss} cusp form $\phi$ admits a Fourier expansion
    \[
        \phi(z) = \sum_{n \in \mathbb{Z}_{\neq 0}} \frac{\lambda_{\phi}(n)}{\sqrt{|n|}} \kappa_{\phi}(ny)e(nx),
    \]
    where $\kappa_{\phi}(y) =2 |y|^{1/2}K_{ir}(2\pi|y|)\mathrm{sgn}(y)^{\frac{1-\delta}{2}}$ with $K_{ir}$ the standard $K$-Bessel function, $\mathrm{sgn}(y)$ the sign function, and $\delta \in \{\pm 1\}$ the $T_{-1}$-eigenvalue of $\phi$. 
    
    We normalize a Hecke--Maa{\ss} cusp form $\phi$ such that $\langle \phi, \phi \rangle_1= 1$ and assume that $\phi$ is  even; otherwise, $\langle \phi, |F_k|^2 \rangle_q = 0$. Define $N_{\Gamma_0(1)}(T) := |\{\phi :|t_\phi| \le T\}|$, the number of Hecke--Maa{\ss} cusp forms such that $|t_\phi| \le T$. By Weyl's law (see, e.g., \cite[Corollary 11.2]{Iw}), we have
        \begin{equation}
        \label{eq:Weyls}
            N_{\Gamma_0(1)}(T) = \frac{T^2}{12} + O(T\log T).
         \end{equation}

 
	\subsection{Eisenstein series}\label{subsec:Eisenstein}
	Let $s \in \mathbb{C}$, $z \in \mathbb{H}$, and $\Gamma_{\infty} = \big\{\pm \big(\begin{smallmatrix} 1 & n\\ & 1 \end{smallmatrix}\big): n \in \z \big\}$. The \textit{real-analytic Eisenstein series} of level 1 is defined as 
 \[
 E(s,z) = \sum_{\Gamma_\infty \bs \Gamma_0(1)} \Imag(\gamma z)^s.
 \]
 The series $E(s,z)$ converges normally for $\Real(s) >1$ and continues meromorphically to  $\Real(s) \ge 1/2$, where $E(s,z)$ has a unique simple pole  at $s=1$ of residue $\res_{s=1} E(s,z) = 3/\pi$.

 Let $\zeta(s)$ be the Riemann zeta function,  $\lambda_s(n) = \sum_{ab=n}(a/b)^s$, and $\theta(s)=\pi^{-s}\Gamma(s)\zeta(2s)$.  The Eisenstein series admits the Fourier expansion 
	\[
	E(s,z) = y^s + \frac{\theta(1-s)}{\theta(s)}y^{1-s} +\frac{2\sqrt{y}}{\theta(s)}\sum_{n \in \mathbb{Z}_{\neq 0}}\tau_{s-1/2}(|n|)K_{s-1/2}(2\pi|n|y)e(nx).
	\]
	For $\Real(s) = 1/2$, we obtain $|\theta(1-s)/\theta(s)| =1$, and such $E(s,z)$ is called a \textit{unitary Eisenstein series}. Let $E_t$ denote $E(1/2+it, \cdot)$.  The work of Huang--Xu \cite[Theorem 1]{HX} shows that for $y\ge 1/2$, if $z = x+iy$, then
	\begin{equation}\label{eq:Huang-xu-bound}
		|E_t(z)|  \ll \sqrt{y}(1+|t|)^{3/8 + \varepsilon}.
	\end{equation}
       

	\subsection{Spectral decomposition}\label{subsec:spectral-decom}
    Let $Y_0(1) = \Gamma_0(1)\bs\h$,  and let $L^2(Y_0(1))$ denote the space of square-integrable  functions on $Y_0(1)$. One can decompose any $\psi \in L^2(Y_0(1))$ into three types of spectra: a constant function, Hecke--Maa{\ss} cusp forms, and unitary Eisenstein series. 
    
	\begin{lem}[Theorem 15.5 of \cite{IK}] Let $\{\phi_j\}_j$ be an orthonormal basis of Hecke--Maa{\ss} cusp forms. For any $t \in \mathbb{R}$, let $E_t$ be a unitary Eisenstein series.  If $\psi\in L^2(Y_0(1))$, then
		\begin{equation*}
			\psi(z) = \frac{3}{\pi}\langle \psi, 1 \rangle_1 + \sum_{j \ge 1}\langle \psi, \phi_j \rangle_1 \phi_j(z) + \frac{1}{4\pi}\int_{\mathbb{R}} \langle \psi, E_t \rangle_1 E_t(z) dt.
		\end{equation*} 
	\end{lem}
	The above decomposition gives us that
 \begin{equation}\label{eq:spect-decom-2}
		\langle \psi, \overline{F_k}G_k \rangle_q =  \frac{3}{\pi}\langle \psi , 1 \rangle_1 \mathds{1}_{f = g} + \sum_{j \ge 1}\langle \psi, \phi_j \rangle_1 \langle \phi_j, \overline{F_k}G_k \rangle_q + \frac{1}{4\pi}\int_{\mathbb{R}} \langle \psi, E_t \rangle_1 \langle E_t, \overline{F_k}G_k \rangle_q dt.
	\end{equation}
        To obtain a rate of convergence for $\langle \psi, \overline{F_k}G_k \rangle_q$, we need to estimate $\langle \phi_j, \overline{F_k}G_k \rangle_q$, $\langle E_t, \overline{F_k}G_k \rangle_q$, $\langle \psi, \phi_j \rangle_1$, and $\langle \psi, E_t \rangle_1$.  For the last two quantities, we recall the following lemma.

	\begin{lem}[Lemma 2.2 of \cite{BH}]\label{lem:petersson-bounds}
		Let $M \ge 1$, $\psi \in C^{\infty}_c(Y_0(1), M)$, and $\phi$ be a Hecke--Maa{\ss} cusp form with spectral parameter $ t_\phi $. If $A\geq 0$, then 
	\[
	       |\langle\psi, \phi\rangle_1| \ll_A \Big(\frac{M}{ 1+|t_\phi| }\Big)^A \quad \text{and} \quad |\langle\psi, E_t\rangle_1| \ll_A \Big(\frac{M}{1+|t|}\Big)^A(1+|t|)^{3/8+\varepsilon}.
	\]
	\end{lem}


	\subsection{Incomplete Poincar\'e series and the Poincar\'e decomposition.}\label{subsec:poincare}
    Recall the notation in Section \ref{sec:intro}. We discuss here the Poincar\'e decomposition, which plays a crucial role when applying Holowinsky's method in Section \ref{subsec: hol-cor} (see more details in \cite[Section 4]{H}).
    
    Let $M \ge 1$ and $\psi \in C^{\infty}_c(Y_0(1), M)$.  Let $\check{\psi}$ be the function on $\h$ to $\mathbb{C}$ such that $\check{\psi}(z) = \psi(z)$ if $z \in \mathcal{F}$, and $\check{\psi}(z) = 0$ otherwise.
    Let $\check{\Psi}(z)$ be the extension of $\check{\psi}(z)$ to $\h$ by $\Gamma_{\infty} = \big\{\pm \big(\begin{smallmatrix} 1 & n\\ & 1 \end{smallmatrix}\big): n \in \z \big\}$. 
   For any $z = x+iy \in \mathbb{H}$,  we denote the $m$-th Fourier coefficients of $\check{\Psi}(z)$ as 
    \[
	\Psi_m(y) = \int_{-1/2}^{1/2} \check{\Psi}(x+iy)e(-mx)\,dx. 
    \]
    By integration by parts, we have that 
	\begin{equation}\label{eq:Psi}
		|\Psi_m(y)| \ll_A (M/|m|)^A.
	\end{equation}
    Then the Fourier expansion of $\check{\Psi}(z)$ can be defined as
    $\check{\Psi}(z) = \sum_{m \in \mathbb{Z}} \Psi_m(y)e(mx).$
    
    For any $\Psi \in C^{\infty}(\mathbb{R})$, a smooth function on $\mathbb{R}$, we define the \textit{incomplete Poincar\'e series} as
	\begin{equation*}
		P_m(z,\Psi) = \sum_{\gamma \in \Gamma_\infty \bs\Gamma_0(1)} e(m\Real(\gamma z))\Psi(\Imag(\gamma z)).
	\end{equation*} 
    When $m=0$, $P_0(z,\Psi)$ is called the \textit{incomplete Eisenstein series}, also denoted by $E(\cdot|\Psi)$. Observe that $\psi(z)=\sum_{ \gamma \in \Gamma_{\infty} \bs\Gamma_0(1)} \check{\Psi}(\gamma z)$. Then we have that
	\begin{equation}\label{eq:poincare}
		\psi(z) = \sum_{ \gamma \in \Gamma_{\infty} \bs\Gamma_0(1)}\sum_{m \in \mathbb{Z}} \Psi_m( \Imag(\gamma z))e(m\Real(\gamma z)) = \sum_{m\in \z}P_m(z, \Psi_m).
	\end{equation}
 
    For any $\Psi \in C^{\infty}(\mathbb{R})$, let $\widehat{\Psi}(s)=\int_0^{\infty}\Psi(y)y^{s-1}\,dy$ 
    denote the Mellin transform of $\Psi(y)$.  
    By unfolding, we have  $\langle E(\cdot|\Psi), 1 \rangle_1 = \widehat{\Psi}(-1)$ and $\langle P_m(\cdot, \Psi), 1 \rangle_1 = 0$. By \eqref{eq:poincare}, it follows that
\begin{multline}\label{eq:poincare2}
        \langle \psi, \overline{F_k}G_k \rangle_q -  \frac{3}{\pi}\langle \psi , 1 \rangle_1 \mathds{1}_{f = g}
        \\ =\langle E(\cdot|\Psi_0), \overline{F_k}G_k \rangle_q - \frac{3}{\pi}\widehat{\Psi}_0(-1)\mathds{1}_{f = g} +\sum_{m \neq 0} \langle P_m(\cdot, \Psi_m), \overline{F_k}G_k \rangle_q.
	\end{multline}

    
    \section{Outline of the proof for Theorem \ref{thm:cor}}
    \label{sec:outline-thm1.1-1.2}

    In this section, we prove Theorem \ref{thm:cor}  assuming the following two key propositions.  The proofs of our key propositions rely on the bounds in Lemmas \ref{lem:bound-Lambda-Maass} and \ref{lem:bound-Lambda-Eis}.
   Define
    \begin{equation}
    \label{eq:def-Q}
      Q = Q(f,g) = \begin{cases}
          q^2, &\mbox{if $f\neq g$ and $q$ is squarefree,}\\
          N_{\ad f}, &\mbox{if $f=g$.} 
      \end{cases}
   \end{equation}
    \subsection{Key propositions}
    To derive Theorem \ref{thm:cor}, we bound $\langle \psi, \overline{F_k}{G_k} \rangle_q -  \tfrac{3}{\pi}\langle \psi , 1 \rangle_1 \mathds{1}_{f=g}$ in two different ways. Our first result, based on Soundararajan's approach in \cite{S}, relies crucially on Theorem \ref{thm:weaksub} and the variant of Watson's triple product formula proved in \cite{N, NPS}.
	\begin{prop}\label{prop:Sound_cor} 
		Keep the notation and hypotheses of Theorem  \ref{thm:cor}. Let $\lambda_f(n)$ and $\lambda_g(n)$ denote the $n$-th Hecke eigenvalues of $f$ and $g$, respectively. If $\varepsilon>0$, then 
		\[
		\langle \psi, \overline{F_k}G_k \rangle_q -  \frac{3}{\pi}\langle \psi , 1 \rangle_1\mathds{1}_{f=g} \ll M^{\delta +\varepsilon}(\log kq)^{\varepsilon}\Big(\frac{q}{\sqrt{Q}}\Big)^{-(\frac{1}{2}-\theta)+\varepsilon}\prod_{p \le k\sqrt{Q}}\Big(1-\frac{\lambda_f(p)^2+\lambda_g(p)^2-1}{2p}\Big),
		\]
        where $\delta = \frac{23}{12}$ when $f = g$ and $\delta = 2$ when  $f \neq g$. 
	\end{prop}

	Our second result is based on Holowinsky's approach in \cite{H}.
	
	\begin{prop}\label{prop:Holow_cor} 
		Keep the notation and hypotheses of Theorem  \ref{thm:cor}. Let $\lambda_f(n)$ and $\lambda_g(n)$ denote the $n$-th Hecke eigenvalues of $f$ and $g$, respectively. 
		If $\varepsilon>0$, then
		\[
		\langle \psi, \overline{F_k}G_k \rangle_q -  \frac{3}{\pi}\langle \psi , 1 \rangle_1\mathds{1}_{f=g} \ll M^{4+\varepsilon}(\log kq)^{\varepsilon}\Big(\frac{q}{\sqrt{Q}}\Big)^{\varepsilon}\prod_{p \le k\sqrt{Q}} \Big(1-\frac{(|\lambda_f(p)|-1)^2+(|\lambda_g(p)|-1)^2}{4p}\Big).
		\]
	\end{prop}

 
	\subsection{Proof of Theorem \ref{thm:cor}}
        \label{subsec:proof_of_thm:cor}
 
		We combine Propositions \ref{prop:Sound_cor}  and \ref{prop:Holow_cor} and follow the optimization technique in Iwaniec's course notes \cite{Iwaniec}. First, observe that 
		\begin{multline*}
			\langle \psi, \overline{F_k}G_k \rangle_q -   \frac{3}{\pi}\langle \psi , 1 \rangle_1\mathds{1}_{f=g}\\ \ll \min\Big \{M^{\delta+\varepsilon}(\log kq)^{\varepsilon}(q/\sqrt{Q})^{-(\frac{1}{2}-\theta)+\varepsilon} \prod_{p \le k\sqrt{Q}}\Big(1-\frac{\lambda_f(p)^2+ \lambda_g(p)^2-1}{2p}\Big),\\ 
			M^{4+\varepsilon}(q/\sqrt{Q})^{\varepsilon} (\log kq)^{\varepsilon}  \prod_{p \le k\sqrt{Q}} \Big(1-\frac{(|\lambda_f(p)|-1)^2+(|\lambda_g(p)|-1)^2}{4p}\Big)\Big\}.
		\end{multline*}
        Let $\alpha \in [0,1]$. If $X,Y>0$, then $\min\{X,Y\}\leq X^{\alpha}Y^{1-\alpha}$. By Deligne's bound, we have that $|\lambda_f(p)|, |\lambda_g(p)| \leq 2$.  Write $\lambda_f = |\lambda_f(p)|$ and $\lambda_g = |\lambda_g(p)|$. Since
        \[
           \max_{\alpha\in[0,1]}\min_{\lambda_f \in[0,2]}\min_{\lambda_g \in[0,2]}\Big(\alpha\Big(\frac{\lambda_f^2+\lambda_g^2-1}{2}\Big)  +(1-\alpha)\Big(\frac{(\lambda_f-1)^2+(\lambda_g-1)^2}{4}\Big)\Big) = \frac{7}{2}-2\sqrt{3},
        \]
    and the maximum is achieved when $\alpha$ equals $2/\sqrt{3}-1$,
   Mertens' theorem implies that
    \begin{equation}
        \label{eq:sound-cor-combine}
        \begin{aligned}
        \langle \psi, \overline{F_k}G_k \rangle_q -   \frac{3}{\pi}\langle \psi , 1 \rangle_1\mathds{1}_{f=g} 
        \ll M^4
        (\log kq)^{\varepsilon}(q/\sqrt{Q})^{-(\frac{2}{\sqrt{3}}-1)(\frac{1}{2}-\theta)+\varepsilon} ({\log(k\sqrt{Q})})^{-(\frac{7}{2}-2\sqrt{3})}.
        \end{aligned}
    \end{equation}
 We invoke the bound $(\log k\sqrt{Q})^{-1}\ll (q/\sqrt{Q})^{\varepsilon}(\log kq)^{-1}$ and the definition of $Q$ in \eqref{eq:def-Q}. Theorem \ref{thm:cor} follows from the bound $(q/q_0)^{1/2} \ll q/\sqrt{Q}$, derived from \cite[Proposition 2.5]{NPS}. 

     When $q=1$ and $f=g$, the exponent of $M$ in \eqref{eq:sound-cor-combine} can be refined to $\tfrac{1}{12}(23-2\sqrt{3})$. To refine Proposition \ref{prop:Sound_cor}, we follow \cite[Lemma 4.4]{LMR} for the contribution from Eisenstein series but slightly improve the bound for $L(\frac{1}{2}+it, f \times f)$ by using $\zeta(\frac{1}{2}+it) \ll (1+|t|)^{13/84+\varepsilon}$.  For the contribution from Hecke--Maa{\ss} cusp forms, we improve the exponent of $|t_\phi|$ in Lemma \ref{lem:bound-Lambda-Maass} using $\mathfrak{C}(\ad f \times \phi) \ll k^4(1+|t_\phi|)^2$ for $|t_\phi| \le k$. Following the proof of Proposition \ref{prop:Sound_cor}, we improve the exponent of $M$ to $17/12+\varepsilon$.
     To refine Proposition \ref{prop:Holow_cor}, we replace our Lemma \ref{lem:H_cor_body} with \cite[Lemma 4.5]{LMR}. Then  the exponent of $M$ in Proposition \ref{prop:Holow_cor} is reduced to $5/3+\varepsilon$.
     Combining the refined bounds as in the beginning of this proof yields the desired result.

    
    \section{Bounds for \texorpdfstring{$L$}{L}-functions}\label{sec:Bound-L-functions}

      In this section, we prove lower bounds for the adjoint lift of a holomorphic newform. We also establish bounds for $L$-functions that will be useful for proving Propositions \ref{prop:Sound_cor} and \ref{prop:Holow_cor}.
     
     Let $\pi \in \mathfrak{F}_2$ be non-dihedral with central character $\omega_{\pi}$. The adjoint $L$-function of $\pi$ is
    \[
        L(s, \ad \pi) = L(s, \pi \times \widetilde{\pi}) \zeta(s)^{-1}.
    \]
    Since $\pi$ is non-dihedral, Gelbert and Jacquet \cite{GJ} shows that  $\ad \pi \in \mathfrak{F}_3$. Thus, $L(s, \ad \pi)$ admits all of the analytic properties of a standard $L$-function in Section \ref{subsec: standardL}.  
    
    If $\pi$ is a holomorphic newform of weight $k$, the gamma factor of $L(s, \ad \pi)$ is defined as 
    \[
        L_{\infty}(s, \ad \pi) = \pi^{-\frac{3s}{2}}\Gamma\Big(\frac{s+1}{2}\Big)\Gamma\Big(\frac{s+k-1}{2}\Big)\Gamma\Big(\frac{s+k}{2}\Big),
    \]
    and if $\pi$ is a Hecke--Maa{\ss} form with eigenvalue $\lambda = 1/4+t^2$, we have
     \[
        L_{\infty}(s, \ad \pi) = \pi^{-\frac{3s}{2}}\Gamma\Big(\frac{s}{2}\Big)\Gamma\Big(\frac{s}{2}+it\Big)\Gamma\Big(\frac{s}{2}+it\Big).
    \]
    Since $\pi$ is non-dihedral, $L(s, \ad \pi)$ has no pole at $s=1$, so the completed $L$-function
    \[
        \Lambda(s,\ad \pi) = N_{\ad \pi}^{s/2}L(s, \ad \pi)L_{\infty}(s, \ad \pi)
    \]
    is entire of order 1, and satisfies the functional equation $\Lambda(s, \ad \pi) = \Lambda(1-s, \ad \pi)$. 
    
    Write $\mathfrak{C}(\pi) = N_{\pi}K_{\pi}$ and $\mathfrak{C}(\ad \pi) = N_{\ad \pi}K_{\ad \pi}$. Then $N_{\ad \pi} = N_{\pi \times \widetilde{\pi}}\mid N_{\pi}^2$ and $K_{\ad \pi} \asymp K_{\pi}$. We now refine the work of Goldfeld, Hoffstein, and Lieman in \cite[Appendix]{HJLP}. 
   
    \begin{thm}
    \label{thm:L(1,adf)}
    Let $\pi \in \mathfrak{F}_2$ be  non-dihedral with central character $\omega_{\pi}$.  There exists an absolute and effectively computable constant $\Cl[abcon]{zfr} >0$ such that $L(s,\ad \pi)\neq 0$ in the region
    \begin{equation}
    \label{eq:ZFR}
    \mathrm{Re}(s)\geq 1-\frac{\Cr{zfr}}{\log(N_{\ad \pi}K_{\pi}(|t|+3))},
    \end{equation} and \begin{equation}\label{eq:l(1,adf)GHL}
    L(1, \ad \pi) \gg (\log N_{\ad \pi}K_{\pi})^{-1}.
     \end{equation}
    \end{thm}
    \begin{proof}
    The zero-free region \eqref{eq:ZFR} holds when $t\neq 0$ by Theorem A.1 in \cite[Appendix]{Humphries}.  When $t=0$, the proof is the same as in \cite[Appendix]{HJLP} apart from the observation that
    \[
        \varphi(s) = \zeta(s) L(s,\ad\pi)^2 L(s,\ad\pi\times\ad\pi) = L(s,(\mathbbm{1}\boxplus \ad\pi)\times(\mathbbm{1}\boxplus \ad\pi)).
    \]
    The analytic conductor of $\varphi(s)$ is $O((N_{\ad\pi}K_{\pi})^8)$, which permits us to express the zero-free region for $L(s,\ad\pi)$ in terms of $N_{\ad\pi}$ instead of $N_{\pi}$.  Lastly, \eqref{eq:l(1,adf)GHL} follows from the proof of the main theorem in \cite[Appendix]{HJLP}, but we replace their zero-free region with \eqref{eq:ZFR}.
  \end{proof}

        Recall the notation in Sections \ref{sec:intro} and \ref{subsec:newforms}. For $f \in \mathcal{B}_k(q)$, we establish a lower bound on $L(1,\ad f)$ that is sensitive to the distribution of Hecke eigenvalues $\lambda_f(p)$ for $p\leq k\sqrt{Q}$.
        
        \begin{lem}
        \label{lem:L(1,adf)}    If $f \in \mathcal{B}_k(q)$ and $Q$ is defined as in \eqref{eq:def-Q}, then 
		\[
		\frac{1}{L(1, \ad f)} \ll (\log\log kq)^{9} \prod_{p \le k\sqrt{Q}}\Big( 1-\frac{\lambda_f(p)^2-1}{p}\Big).
		\]
	\end{lem}
	\begin{proof} 
		Let $X \ge 2$. By partial summation, we obtain
		\begin{equation}\label{eq:lemma2-HS-2}
			\sum_{n \ge X}\frac{\Lambda_{\ad f}(n)}{n \log n} = -\Big(\sum_{n\leq X}
			\Lambda_{\ad f}(n)\Big) \frac{1}{X\log X} +\int_X^{\infty} \Big(\sum_{n\le t} \Lambda_{\ad f}(n)\Big) \frac{\log t+1}{t^2(\log t)^2} \, dt.
		\end{equation}
        By \cite[Theorem 2.4]{KT} and Theorem \ref{thm:L(1,adf)}, 
		there exist absolute and effectively computable constants $\Cl[abcon]{PNT-adf-range}>0$ and $\Cl[abcon]{PNT-adf-error}>0$ such that if $X \ge (k^2Q)^{\Cr{PNT-adf-range}}$, then 
		\[
		\sum_{n \le X} \Lambda_{\ad f}(n) \ll X\exp\Big(-\Cr{PNT-adf-error} \frac{\log X}{\log(k^2Q) +\sqrt{\log X}}\Big).
		\]
	 Applying the above bound in \eqref{eq:lemma2-HS-2} and choosing $X = (k^2Q)^{(6/\Cr{PNT-adf-error})\log\log (k^2Q)}$, we obtain
    \begin{equation}\label{eq:lemma2-HS-1}
    \begin{aligned}
			\log L(1, \ad f) = \sum_{n \ge 2} \frac{\Lambda_{\ad f}(n)}{n \log n} = \sum_{2 \le n \le X}\frac{\Lambda_{\ad f}(n)}{n \log n} + O(1) =\sum_{p \le X} \frac{\lambda_{\ad f}(p)}{p}+O(1),
   \end{aligned}
    \end{equation}
    using Deligne's bound to trivially bound the contribution from the composite $n\leq X$.  
    
    By the Hecke relations, if $p\nmid q$, then $\lambda_{\ad f}(p)=\lambda_f(p^2)=\lambda_f(p)^2-1$.  Applying this fact and another application of Deligne's bound in \eqref{eq:lemma2-HS-1}, we obtain that
    \begin{align*}
    \log L(1,\ad f)&=\sum_{p\leq k\sqrt{Q}}\frac{\lambda_f(p)^2-1}{p}+\sum_{\substack{p\leq k\sqrt{Q},\, p\mid q}}\frac{\lambda_{\ad f}(p)-\lambda_f(p)^2+1}{p} + \sum_{k\sqrt{Q} \le p \le X} \frac{\lambda_{\ad f}(p)}{p}+O(1)\\
    &\geq \sum_{p\leq k\sqrt{Q}}\frac{\lambda_f(p)^2-1}{p}-\sum_{p\mid q}\frac{6}{p} - \sum_{k\sqrt{Q} \le p \le X} \frac{3}{p}+O(1).
    \end{align*}
    The lemma now follows from Mertens' theorem.
    \end{proof}

In the next two lemmas, we establish some useful bounds for Section \ref{subsec: Sound-cor}.  

\begin{lem}\label{lem:bound-Lambda-Maass}
 Recall the notation in Sections \ref{subsec:newforms} and \ref{subsec:maasscuspforms}. Let $\phi$ be a Hecke--Maa{\ss} cusp form with spectral parameter $t_\phi $. Let $f,g \in \mathcal{B}_k(q)$ and $Q$ be defined as in \eqref{eq:def-Q}. If $\varepsilon > 0$, then
\begin{equation*}
    \frac{\Lambda(\tfrac{1}{2}, f\times g \times \phi)}{\Lambda(1,\ad \phi)\Lambda(1, \ad f)\Lambda(1, \ad g)} \ll \frac{Q^{\frac{1}{2}}(1+|t_\phi|)^{2\delta_1+\varepsilon}}{(\log kQ)^{1-\varepsilon}L(1,\ad f)L(1,\ad g)},
\end{equation*}
where $\delta_1 = 11/12$ when $f=g$ and $\delta_1 = 1$ otherwise.
\end{lem}

\begin{proof}
    By Stirling's formula, it follows that 
    \begin{equation}\label{eq:stirling-cor}
     \frac{L_\infty(\tfrac{1}{2}, f \times g \times \phi)}{L_\infty(1, \ad \phi) L_\infty(1,\ad f) L_\infty(1,\ad g)} 
        \ll \frac{1}{k}.
    \end{equation}
We first consider when $f = g$. Since $f$ has trivial nebentypus, $L(\frac{1}{2}, f \times {f} \times \phi) = L(\frac{1}{2}, \ad f\times \phi)L(\frac{1}{2}, \phi)$. We apply Theorem \ref{thm:weaksub} for the first factor and invoke  $L(\frac{1}{2}, \phi) \ll  (1+|t_\phi|) ^{1/3+\varepsilon/3}$ from \cite[Corollary 2]{I} for the second factor.
Putting together the above bounds, the bound $L(1,\ad \phi)^{-1} \ll \log(1+|t_\phi|)$ from \cite{HJLP}, and \eqref{eq:stirling-cor}, we obtain the lemma when $f=g$.
If $f \neq g$, we proceed similarly but directly apply Theorem \ref{thm:weaksub} with $L(\frac{1}{2}, f \times g \times \phi)$.
    \end{proof}

    \begin{lem}\label{lem:bound-Lambda-Eis}
    Recall the notation in Section \ref{subsec:newforms}. Let $f,g \in \mathcal{B}_k(q)$ and $Q$ be defined as in \eqref{eq:def-Q}. For any $s \in \mathbb{C}$, let $\xi(s)$ be the completed zeta function. If $t \in \mathbb{R}$ and $\varepsilon > 0$, then
    \begin{equation*}
        \frac{\Lambda(\frac{1}{2}+it, f \times g)^2}{\xi(1+2it)^2\Lambda(1, \ad f)\Lambda(1, \ad g) } \ll  \frac{Q^{\frac{1}{2}}(1+|t|)^{2\delta_2+\varepsilon}}{(\log kQ)^{2-\varepsilon}L(1,\ad f)L(1,\ad g)},
    \end{equation*}
    where $\delta_2 = 19/21$ if $f=g$ and $\delta_2 = 1$ otherwise. 
    \end{lem}
    \begin{proof}
    By Stirling's formula,  it follows that 
    \begin{equation}\label{eq:stirling-cor-2}
        \begin{aligned}
        \frac{L_\infty(\tfrac{1}{2} +it, f \times g)^2}{\Gamma(\frac{1}{2}+it)^2 L_\infty(1,\ad f)L_\infty(1,\ad g)} 
        \ll \frac{1}{k}.
        \end{aligned}
    \end{equation}
   We first consider when $f=g$. Similarly to Lemma \ref{lem:bound-Lambda-Maass}, we have $L(\tfrac{1}{2}+it, f \times f) = L(\tfrac{1}{2}+it, \ad f)\zeta(\tfrac{1}{2}+it)$. We apply \cite[Theorem 1]{S} for the first factor and invoke the bound  $\zeta(1/2+it) \ll_{\varepsilon} (1+|t|)^{13/84+\varepsilon/3}$ from \cite{BJ} for the second factor. Putting together the above bounds, the bound $\zeta(1+2it)^{-1} \ll \log(1+|t|)$, and \eqref{eq:stirling-cor-2}, we obtain the lemma when $f=g$.   
    
   If $f\neq g$, we proceed similarly but directly apply \cite[Theorem 1]{S} with $L(\frac{1}{2}+it, f \times g)$.
    \end{proof}

	\section{Effective correlation and decorrelation}\label{sec:correlation}

\subsection{Soundararajan's approach}\label{subsec: Sound-cor} 
   We follow the idea of Nelson, Pitale, and Saha in \cite{NPS} to estimate $\langle\phi, \overline{F_k}G_k \rangle_q$ and $\langle E_t, \overline{F_k}G_k  \rangle_q$. This allows us to prove Proposition \ref{prop:Sound_cor} using \eqref{eq:spect-decom-2}.
   \begin{lem}\label{lem:Sound_cor} Keep the notation in Theorem \ref{thm:cor} and \eqref{eq:def-Q}. Let $\phi$ be a Hecke--Maa{\ss} cusp form with spectral parameter $ t_\phi $. For any $t \in \mathbb{R}$, let $E_t$ be the unitary Eisenstein series.  
    If $\varepsilon > 0$, then 
		\begin{equation}\label{eq:sound_maass}
			\langle \phi, \overline{F_k}G_k \rangle_q \ll \frac{(1+|t_{\phi}|)^{\delta_1+\varepsilon}}{L(1, \ad f)^{\frac{1}{2}}L(1, \ad g)^{\frac{1}{2}}}\frac{(q/\sqrt{Q})^{-\frac{1}{2}+\theta+\varepsilon}}{(\log kQ)^{\frac{1}{2}-\varepsilon}},
		\end{equation}
    where $\delta_1 = 11/12$ if $f=g$ and $\delta_1 = 1$ otherwise, 
		and 
		\begin{equation}\label{eq:sound_Eis}
			\langle E_t, \overline{F_k}G_k \rangle_q \ll \frac{(1+|t|)^{\delta_2+\varepsilon}}{L(1, \ad f)^{\frac{1}{2}}L(1, \ad g)^{\frac{1}{2}}}\frac{(q/\sqrt{Q})^{-\frac{1}{2}+\varepsilon}}{(\log kQ)^{1-\varepsilon}},
		\end{equation}
        where $\delta_2 = 19/21$ if $f=g$ and $\delta_2 = 1$ otherwise.
	\end{lem}

    \begin{proof} 
    Recall Ichino's generalization of Watson's formula. Denote $I_v^*$ the corresponding normalized local Ichino integral (see, e.g., \cite[Theorem 3.1]{NPS} and \cite[Remark 4.2]{N}). We have 
		\begin{equation}\label{eq:Watson-maass}
			\frac{|\langle \phi, \overline{F_k}G_k \rangle_q|^2}{\langle \phi, \phi \rangle_1 \langle 1, |F_k|^2 \rangle_q \langle 1, |G_k|^2 \rangle_q} = \frac{1}{8}\frac{\Lambda(\tfrac{1}{2}, f \times g \times \phi)}{\Lambda(1, \ad \phi)\Lambda(1,\ad f)\Lambda(1,\ad g)} \prod_v I_v^*.
		\end{equation}
    For any prime $p$, let $n_p = \ord_p(q)$ and $m_p = \ord_p({q}/{\sqrt{Q}})$. The values of  $I_v^*$  are defined and calculated explicitly in \cite[Section 2]{NPS} and \cite[Section 4]{N}. In particular, $I_\infty^* = 1$, and 
		\begin{equation}\label{eq:I_p*} I_p^*
			\begin{cases}
				= 1, &\text{ if } p \nmid q,\\
				< 10^5p^{-n_p}\tau(p^{m_p})^2p^{2\theta m_p}, &\text{ if } p \mid  q \text{ and } f=g,\\
                = p^{-1}, &\text{ if } p \mid  q \text{ and } f\neq g.
			\end{cases}
		\end{equation}
        
		Using the above calculation of $I_v^*$, we obtain that 
        \begin{equation}\label{eq:I_v*-approx}
            \prod_v I_v^* \ll 10^{5\omega({q}/{\sqrt{Q}})}q^{-1}\tau({q}/{\sqrt{Q}})^2({q}/{\sqrt{Q}})^{2\theta} \ll q^{-1}({q}/{\sqrt{Q}})^{2\theta+2\varepsilon},
         \end{equation}
        using $10^{5\omega(q/\sqrt{Q})}\tau(q/\sqrt{Q})^2 \ll (q/\sqrt{Q})^{2\varepsilon}$.
        Recall that $\langle \phi, \phi \rangle_1 = \langle 1, |F_k|^2 \rangle_q = \langle 1, |G_k|^2 \rangle_q = 1$. Inserting \eqref{eq:I_v*-approx} and Lemma \ref{lem:bound-Lambda-Maass} into \eqref{eq:Watson-maass} and taking the square root, we obtain that 
\begin{equation}\label{eq:watson-cor1}
		  \langle \phi, \overline{F_k}G_k \rangle_q \ll \frac{Q^{1/4}(1+|t_{\phi}|)^{\delta_1+\varepsilon}}{(\log kQ)^{1/2-\varepsilon}L(1,\ad f)^{1/2}L(1,\ad g)^{1/2}} q^{-1/2}({q}/{\sqrt{Q}})^{\theta+\varepsilon},
		\end{equation}
    where $\delta_1 = 11/12$ when $f=g$ and $\delta_1 = 1$ otherwise. Hence, \eqref{eq:sound_maass} follows.

		 Let $\xi(s)$ be the completed Riemann zeta function. Similarly to \eqref{eq:Watson-maass}, we have
    \begin{equation}\label{eq:watson-cor-2}
			|\langle E_t, \overline{F_k}G_k \rangle_q|^2 =  \frac{|\Lambda(1/2+it, f\times g)|^2}{|\xi(1+2it)|^2\Lambda(1,\ad f)\Lambda(1,\ad g)}\prod_v J_v^*, 
		\end{equation}
     where $J_\infty^* = 1$ and $J_p^*$ are the same as in \eqref{eq:I_p*}, except $J_p^* < 10^5p^{-n_p}\tau(p^{m_p})^2$ when $p \mid q$ and $f=g$ (see more details in \cite[Remark 3.2]{NPS}). By these calculation of $J_v^*$, we have $\prod_v J_v^* \ll q^{-1}({q}/{\sqrt{Q}})^{2\varepsilon}$. Hence, \eqref{eq:sound_Eis} follows from Lemma \ref{lem:bound-Lambda-Eis}.
    \end{proof}
    \begin{rek}    
     We make the assumption that $q$ is squarefree when $f \neq g$ because to the author's knowledge, an analogue of \eqref{eq:Watson-maass} that ensures a bound of the shape 
    \[
        \langle \phi, \overline{F_k}G_k \rangle_q^2 \ll \frac{\Lambda(1/2, f \times g \times \phi)}{\Lambda(1, \ad \phi)\Lambda(1,\ad f)\Lambda(1,\ad g)}
    \]
    is only known when $q$ is squarefree.  See Hu \cite[Section 3]{YH} for partial progress in this direction. 
    \end{rek}

    \begin{proof}[Proof of Proposition \ref{prop:Sound_cor}] 
    We recall \eqref{eq:spect-decom-2} and follow the proof of \cite[Proposition 3.3]{H}. 
    First, we estimate the sums when $ |t_\phi|  \ge (kQ)^\varepsilon$ and $|t| \ge (kQ)^\varepsilon$. 
    We repeat the proofs of Lemmas \ref{lem:bound-Lambda-Maass} and \ref{lem:bound-Lambda-Eis}, but  with the convexity bound. It follows from \eqref{eq:Watson-maass}, \eqref{eq:watson-cor-2},   and \eqref{eq:l(1,adf)GHL} that
		\begin{equation}
			\label{eq:sound-main-prop-1}
			\begin{aligned}
                \langle \phi, \overline{F_k}G_k \rangle_q  &\ll  (1+|t_{\phi}|) ^{\delta_1+\varepsilon}(q/\sqrt{Q})^{-1/2+\theta+\varepsilon}\log(kQ),\\ 
                \langle E_t, \overline{F_k}G_k \rangle_q 
                &\ll (1+|t|)^{\delta_2+\varepsilon}(q/\sqrt{Q})^{-1/2+\varepsilon}\log(kQ).
			\end{aligned}
		\end{equation}
    By Lemma \ref{lem:petersson-bounds} and \eqref{eq:sound-main-prop-1}, for any large $A > 0$, we have that
\begin{multline}\label{eq:sound-main-prop-tail}
            \sum_{ |t_\phi|  \ge (kQ)^{\varepsilon}} \langle \psi, \phi \rangle_1 \langle \phi, \overline{F_k}G_k \rangle_q + \frac{1}{4\pi}\int_{|t|\ge (kQ)^{\varepsilon}} \langle \psi, E_t \rangle_1 \langle E_t, \overline{F_k}G_k \rangle_q dt \\
            \ll_{A}(q/\sqrt{Q})^{-1/2+\theta+\varepsilon}(kQ)^{-A}.
		\end{multline}

		Now, we consider  when $|t|, |t_\phi|  \in [M (\log kQ)^{\varepsilon/4}, (kQ)^\varepsilon]$. By Lemma \ref{lem:Sound_cor} and \eqref{eq:l(1,adf)GHL}, we have 
		\begin{equation}
			\label{eq:sound-main-prop-tail-2}
			\begin{aligned}
				\langle \phi, \overline{F_k}G_k \rangle_q &\ll  (1+|t_{\phi}|) ^{\delta_1+\varepsilon} (q/\sqrt{Q})^{-1/2+\theta+\varepsilon}(\log kQ)^{1/2+{\varepsilon}/{2}},\\
				\langle E_t, \overline{F_k}G_k \rangle_q &\ll (1+|t|)^{\delta_2+\varepsilon} (q/\sqrt{Q})^{-1/2+\varepsilon}(\log kQ)^{{\varepsilon}/{2}}.
			\end{aligned}
		\end{equation}
       By Lemma \ref{lem:petersson-bounds}, \eqref{eq:Weyls}, and  \eqref{eq:sound-main-prop-tail-2}, we obtain that for any sufficiently large $A > 0$
        \begin{multline}
        \label{eq:sound-main-prop-tail-4}
			\sum_{ |t_\phi|  \in [M(\log kQ)^{\varepsilon/4}, (kQ)^\varepsilon]}\langle \psi, \phi \rangle_1 \langle \phi, \overline{F_k}G_k \rangle_q + \frac{1}{4\pi}\int_{|t| \in [M(\log kQ)^{\varepsilon/4}, (kQ)^\varepsilon]} \langle \psi, E_t \rangle_1 \langle E_t, \overline{F_k}G_k \rangle_q dt\\
             \ll_A M^{2+\delta_1+\varepsilon}(q/\sqrt{Q})^{-1/2+\theta+\varepsilon}(\log kQ)^{-A}.
		\end{multline}
  
    We now bound the contribution from the ranges $|t_\phi|, |t| \le M(\log kQ)^{\varepsilon/4}$. By the Parseval formula, we have that  $\sum_{|t_\phi| \le M(\log kQ)^{\varepsilon/4}} |\langle\psi,\phi\rangle_1|^2 \le \lVert \psi\rVert_2^2 \ll 1$.  Applying the Cauchy--Schwarz inequality and Lemma \ref{lem:Sound_cor}, we obtain that
		\begin{equation*}
            \label{eq:parsevel-2}
			\begin{aligned}
        \sum_{ |t_\phi|  \le M(\log kQ)^{\varepsilon/4}} \langle \psi, \phi \rangle_1 \langle \phi, \overline{F_k}G_k \rangle_q
        \ll \Big(\sum_{ |t_\phi|  \le M(\log kQ)^{\varepsilon/4}} \frac{ (1+|t_\phi|)^{2\delta_1+2\varepsilon}(q/\sqrt{Q})^{-1+2\theta+2\varepsilon}}{L(1, \ad f)L(1, \ad g)(\log kQ)^{1-\varepsilon}}\Big)^{1/2}.
	\end{aligned}
		\end{equation*}
        By Lemma \ref{lem:L(1,adf)}, \eqref{eq:Weyls}, and the fact that $\prod_{p\le x}(1-\delta/{p}) \asymp (\log x)^{-\delta}$, the above bound equals
        \begin{equation}\label{eq:main-maass}
        \ll M^{\delta_1+1+\varepsilon}(\log kq)^{\varepsilon}\Big(\frac{q}{\sqrt{Q}}\Big)^{-1/2+\theta+\varepsilon}\prod_{p \le k\sqrt{Q}}\Big(1-\frac{\lambda_f(p)^2+\lambda_g(p)^2-1}{2p}\Big).
        \end{equation}
	For Eisenstein series, the Cauchy--Schwarz inequality, Lemmas  \ref{lem:L(1,adf)} and \ref{lem:Sound_cor}  gives
		\begin{multline}\label{eq:main-eis}
				\int_{|t| \le M(\log kQ)^{\varepsilon/4}} \langle \psi, E_t \rangle_1 \langle E_t, \overline{F_k}G_k \rangle_q \, dt\\ 
                \ll M^{\delta_2+\frac{1}{2}+\varepsilon}(\log kq)^{\varepsilon}\Big(\frac{q}{\sqrt{Q}}\Big)^{-1/2+\varepsilon}\prod_{p \le k\sqrt{Q}}\Big(1-\frac{\lambda_f(p)^2+\lambda_g(p)^2}{2p}\Big).
		\end{multline}
	The proposition then follows upon combining \eqref{eq:sound-main-prop-tail}, \eqref{eq:sound-main-prop-tail-4}, \eqref{eq:main-maass}, and \eqref{eq:main-eis}.
	\end{proof}

	\subsection{Holowinsky's approach}\label{subsec: hol-cor} To prove Proposition \ref{prop:Holow_cor}, we first recall the standard unfolding method  for $\langle P_m(\cdot, \Psi_m), |F_k|^2  \rangle_q$. See, for example, \cite[p. 2398]{YH} for detailed calculation.
    \begin{lem}
    \label{lem:unfolding} Recall the notation in Section \ref{subsec:poincare}.
         For any $m \in \mathbb{Z}$ and any $\Psi \in C^{\infty}_c(\mathbb{R}_{\ge 0})$, a smooth compactly supported function on the set of nonnegative real numbers, we have 
        \[
        \langle P_m(\cdot, \Psi), |F_k|^2 \rangle_q = \sum_{\mathfrak{a}_j \in \mathcal{C}} \int_{y_j=0}^{\infty} \Psi(w_jy_j) \int_{x_j = 0}^{1} e(mw_jx_j)|F_k(z)|^2y_j^{-2}dx_j\,dy_j.
        \]
    \end{lem}
    \begin{rek}\label{rek:unfolding}   
        For a squarefree $q$, as discussed in Section \ref{subsec:cusps} and by Lemma \ref{lem:unfolding}, we obtain that
        \begin{equation}\label{eq:unfolding-2}
			\langle P_m(\cdot, \Psi), \overline{F_k}G_k \rangle_q = 
            \sum_{d|q} \int_{y=0}^{\infty} \Psi(dy) \int_{x = 0}^{1} e(mdx)F_k(z)\overline{G_k(z)} y^{-2}dx\,dy.
		\end{equation} 
    \end{rek}

	Recall the notation in Section \ref{subsec:poincare} and \eqref{eq:poincare2}. To prove Proposition \ref{prop:Holow_cor}, we split the sum of the contributions from incomplete Poincar\'e into two ranges: $0<|m|<M(\log kQ)^{\varepsilon/3}$ and  $|m|\ge M(\log kQ)^{\varepsilon/3}$. To  bound the sum in each range, we recall the following propositions. 

	\begin{prop}\label{prop:shiftedcon}
		Recall the notation in Theorem \ref{thm:cor} and Section \ref{subsec:cusps}. Let $m \neq 0$ be an integer, $x \ge 1$ be a real number, and $\varepsilon \in (0,1)$. Let $q_\diamond^2$ be the largest square divisor of $q$. Let $c$ be a positive divisor of $q$ and $\lambda_{[c]}$ be defined as in \eqref{eq:lambdac}. If $f = g$, then 
		\[
		\sum_{\substack{n_1 \ge 1, n_2 = n_1+m \ge 1\\ \max(n_1,n_2) \le x}} |\lambda_{[c]}(n_1)\lambda_{[c]}(n_2)| \ll q_\diamond^\varepsilon (\log\log(e^eq))^{O(1)}\frac{x\prod_{p\le x}(1+2|\lambda_f(p)|/p)}{(\log ex)^{2-\varepsilon}}.
		\]
  If $f \neq g$, then
        \[
		\sum_{\substack{n_1 \ge 1, n_2 = n_1+m \ge 1 \\ \max(n_1,n_2) \le x}} |\lambda_{f}(n_1)\lambda_{g}(n_2)| \ll \frac{x\prod_{p\le x}(1+(|\lambda_f(p)|+ |\lambda_g(p)|)/p)}{(\log ex)^{2-\varepsilon}}.
		\]
	\end{prop}
        \begin{proof} The first statement is from \cite[Proposition~3.17]{NPS}. The proof of the second statement follows directly from the proof of \cite[Theorem~3.10]{N}.
        \end{proof}

	\begin{prop}[Proposition 3.18 of \cite{NPS}]\label{prop:shiftedcon2}
		Let $x \ge 2$ be a real number and $q$ be a positive integer. If $\delta \in (0,1)$, then 
		\[
		\sum_{c|q} \frac{[q/c^2,1]\varphi((c,q/c))}{\log([q/c^2,1]x)^{2-\delta }}\ll \frac{q (\log\log(e^eq))^{O(1)}}{(\log qx)^{2-\delta }}.
		\]
	\end{prop}
	
	We first estimate the sum of the incomplete Poincar\'e series when $|m| \ge M(\log kQ)^{\varepsilon/3}$.
	
	\begin{lem}\label{lem:H_cor_tail} 
    If $A > 0$ and $\varepsilon \in (0,1)$ are real numbers, then we have
		\[
			\sum_{|m| \ge M(\log kQ)^{\varepsilon/3}} \langle  P_m(\cdot, \Psi_m), \overline{F_k}G_k \rangle_q \ll_{A,\varepsilon} M(\log kQ)^{-A}.
		\]
	\end{lem}
	\begin{proof}
        Since $\supp \check{\psi} \subset \mathcal{F}$, the function $\Psi(\Imag(\gamma z))$ is nonzero only if $\Imag(\gamma z) \ge \sqrt{3}/2$. By  \cite[Lemma 2.10]{Iw}, we have $\#\{\gamma \in \Gamma_\infty\bs\Gamma_0(1): \Imag(\gamma z) \ge \sqrt{3}/2\} \ll 1$.
        By \eqref{eq:Psi} and the fact that $|F_k(z)||G_k(z)| \ll |F_k(z)|^2+ |G_k(z)|^2$, we obtain that $\langle P_m(\cdot, \Psi_m), \overline{F_k}G_k \rangle_q \ll_A (M/|m|)^A.$
        Hence, the lemma follows upon summing over $m \ge M(\log kQ)^{\varepsilon/3}$.
 \end{proof}
 
    Next, we estimate the sum of the incomplete Poincar\'e series over the range $1 \le |m| \le M (\log kQ)^{\varepsilon/3}$.
    With the notation in Section \ref{subsec:poincare}, we know that $\check{\psi}(z) = \psi(z)|_{\mathcal{F}}$ and $\check{\Psi}$ is the extension of $\check{\psi}$ to $\mathbb{H}$ by $\Gamma_{\infty}$. 
    Since $\psi \in C_c^{\infty}(Y_0(1),M)$, we have that for $M \ge 1$
	\begin{equation}
        \label{eq:derivative-Psi_m}
		y^{\max(j, 1/2)}\Psi_m^{(j)}(y)= \int_{-1/2}^{1/2} y^{\max(j, 1/2)}\frac{\partial^{j}}{\partial y^{j}}\check{\Psi}(x+iy)e(-mx)\,dx \ll_j M^j.
	\end{equation}
	\begin{lem}\label{lem:H_cor_body}
		Recall the notation in Section \ref{subsec:poincare} and the estimate in \eqref{eq:derivative-Psi_m}. 
        Define
        \[
            M_f(x) := \frac{\prod_{p\le x}(1+2|\lambda_f(p)|/p)}{L(1, \ad f)(\log ex)^2},
	\] 
        and let $q_\diamond^2$ denote the largest square divisor of $q$. If $\varepsilon\in (0,1)$, then we have
		\[
		\sum_{1\leq |m|\le M(\log kQ)^{\varepsilon/3}} \langle P_m(\cdot, \Psi_m), \overline{F_k}G_k \rangle_q \ll_{\varepsilon} M^{4+\varepsilon}(\log kq)^\varepsilon q_\diamond^\varepsilon M_f(kq)^{1/2}M_g(kq)^{1/2}.
		\]
	\end{lem}
	\begin{proof} We first consider when $f=g$. 
		Recall that $F_k(z) = \rho_f(1)y^{k/2}f(z)$ and $|f(z)|^2y^k = |f|_k \sigma_j|^2(z_j)y_j^k$ as discussed in Sections \ref{subsec:newforms} and \ref{subsec:cusps}. By Lemma \ref{lem:unfolding} and \eqref{eq:fourier-zj}, we have
        \begin{multline}\label{eq:body-hol-cor}
				\langle P_m(\cdot, \Psi_m), |F_k|^2 \rangle_q 
				= |\rho_f(1)|^2 \sum_{\mathfrak{a}_j \in \mathcal{C}}  \sum_{\substack{n_1\ge 1\\ n_2 =n_1+mw_j \ge 1}}\lambda_j(n_1)\lambda_j(n_2) (n_1n_2)^{\frac{k-1}{2}}\\
                \cdot \int_{0}^{\infty} \exp(-2\pi(n_1+n_2)y_j)\Psi_m(w_jy_j)y_j^{k-2}\,dy_j.
		\end{multline}
		
		We use the method in  \cite[Section 2]{LS} to bound the integral in \eqref{eq:body-hol-cor}. Define the Mellin transform of $\Psi_m$ as $\widehat{\Psi}_m(s) = \int_0^{\infty}\Psi_m(y)y^{s-1}\,dy$. Since $\Psi_m \in \mathcal{C}_c^{\infty}(\mathbb{R}_{\ge 0})$, $\widehat{\Psi}_m$ is entire. For any integer $j \ge 0$ and $\Real(s)<0$,  integrating by parts $j$ times and applying  \eqref{eq:derivative-Psi_m}, we have that 
\begin{equation}\label{eq:rapiddecay}
		\widehat{\Psi}_m(s) \ll_j M^j/(1+|s|)^j, 
		\end{equation}
		and \eqref{eq:rapiddecay} holds for any real $j \ge 0$. We have the Mellin inversion of $\widehat{\Psi}_m(s)$: for any $A,x >0$
        \begin{equation}\label{eq:Mellin-inversion}
			\Psi_m(x) = \frac{1}{2\pi i} \int_{(-A)}\widehat{\Psi}_m(s)x^{-s} ds = \frac{1}{2\pi i} \int_{(A)}\widehat{\Psi}_m(-s)x^{s} ds. 
		\end{equation}
  
    Applying the Mellin inversion formula, we obtain that the integral in \eqref{eq:body-hol-cor} becomes 
		\begin{multline}\label{eq:body-hol-cor-2}
				\int_{(A)}\widehat{\Psi}_m(-s)\int_0^\infty y_j^{s+k-2}w_j^s\exp(-2\pi(n_1+n_2)y_j)\,dy_j\,\frac{ds}{2\pi i} \\
				 =  \frac{\Gamma(k-1)}{(4\pi)^{k-1}} \Big(\frac{1}{\frac{n_1+n_2}{2}}\Big)^{k-1}\int_{(A)}\frac{\widehat{\Psi}_m(-s)w_j^s}{(4\pi(\frac{n_1+n_2}{2}))^{s}}\frac{\Gamma(s+k-1)}{\Gamma(k-1)}\,\frac{ds}{2\pi i}.
		\end{multline} 
        By Stirling's formula, if $\Real(s)=A > 0$, then we have that $\frac{\Gamma(s+k-1)}{\Gamma(k-1)} \ll_A k^{A}(1+|s|)^2$ (see \cite[Eq.~(2.3)]{LS}). Then, by using \eqref{eq:rapiddecay} with $j = 3+\varepsilon$, we conclude that the integral in \eqref{eq:body-hol-cor-2} is
		\begin{equation}\label{eq:applyingMellinBound}
			\ll_{A, \varepsilon} \int_{(A)}\frac{M^{3+\varepsilon}}{(1+|s|)^{3+\varepsilon}}\frac{w_j^Ak^A}{(4\pi(\frac{n_1+n_2}{2}))^{A}}(1+|s|)^2\,|ds| \ll_A M^{3+\varepsilon} \max\Big(1, \frac{\max(n_1,n_2)}{w_jk}\Big)^{-A}.
		\end{equation}
        Hence, using \eqref{eq:body-hol-cor-2} and \eqref{eq:applyingMellinBound} in \eqref{eq:body-hol-cor},
        we conclude that
        \begin{multline}\label{eq:body-hol-cor-3} 
        \langle P_m(\cdot, \Psi_m), |F_k|^2 \rangle_q \ll_{A} \frac{M^{3+\varepsilon}|\rho_f(1)|^2\Gamma(k-1)}{(4\pi)^{k-1}} \\  \cdot \sum_{\mathfrak{a}_j \in \mathcal{C}}  \sum_{\substack{n_1\ge 1\\ n_2 =n_1+mw_j \ge 1}}|\lambda_j(n_1)\lambda_j(n_2)
        |\Big(\frac{\sqrt{n_1n_2}}{\frac{n_1+n_2}{2}}\Big)^{k-1} \max\Big(1, \frac{\max(n_1,n_2)}{w_jk}\Big)^{-A}.
		\end{multline}
	
        By \eqref{eq:rho_f(1)}, \eqref{eq:regrouping}, and the inequality of arithmetic and geometric means, \eqref{eq:body-hol-cor-3} is 
        \begin{equation}\label{eq:body-hol-cor-4}
			\ll_{A} \frac{M^{3+\varepsilon}}{L(1,\ad f)kq} \sum_{c\mid q}\#\mathcal{C}[c]  \sum_{\substack{n_1 \ge 1\\ n_2 = n_1+mw_c \ge 1}}\lambda_{[c]}(n_1)\lambda_{[c]}(n_2) \max\Big(1, \frac{\max(n_1,n_2)}{w_ck}\Big)^{-A}.
		\end{equation}
         By decomposing dyadically and using Proposition \ref{prop:shiftedcon} and $A =2$, the inner sum of \eqref{eq:body-hol-cor-4} is
		\begin{equation}
            \label{eq:body-hol-cor-5}
		\ll q_\diamond^\varepsilon (\log\log(e^eq))^{O(1)} \frac{w_ck\prod_{p\le w_ck}(1+2|\lambda_f(p)|/p)}{\log(ew_ck)^{2-\frac{\varepsilon}{3}}}. 
		\end{equation}
        Inserting \eqref{eq:body-hol-cor-5} into \eqref{eq:body-hol-cor-4} and using $w_c = [q/c^2, 1]$ and $\#\mathcal{C}[c] = \varphi((c,q/c))$, we obtain that
		\begin{equation*}
			\begin{aligned}
				\langle P_m(\cdot, \Psi_m), |F_k|^2 \rangle_q 
			\ll M^{3+\varepsilon}\frac{q_\diamond^\varepsilon (\log\log(e^eq))^{O(1)}}{L(1,\ad f)q}\sum_{c\mid q} \frac{\varphi((c,q/c))[q/c^2, 1]}{(\log([q/c^2, 1]k))^{2-\frac{\varepsilon}{3}}}\prod_{p\le kq}(1+2|\lambda_f(p)|/p).
			\end{aligned}
		\end{equation*}
	Applying Proposition \ref{prop:shiftedcon2} and the definition of $M_f(x)$, we conclude that 
    \begin{equation*}
        \begin{aligned}
        \langle P_m(\cdot, \Psi_m), |F_k|^2 \rangle_q
        \ll M^{3+\varepsilon}(\log kq )^{\frac{2\varepsilon}{3}}q_\diamond^\varepsilon M_f(kq).
        \end{aligned}
    \end{equation*}
    The lemma then follows by summing the above estimate over $m$ in $[1, M(\log kQ)^{\varepsilon/3}]$. 
    
   If $f \neq g$, we proceed similarly,  except we replace Lemma \ref{lem:unfolding} with \eqref{eq:unfolding-2} in the proof. 
   \end{proof}

    To estimate the contribution from the incomplete Eisenstein series, we first write the Fourier expansion
    $
        E(z|\Psi_0) = a_{\Psi_0,0}(y) + \sum_{|l|\ge 1} a_{\Psi_0, l}(y)e(lx)$ and recall the following lemma.
	\begin{lem}[Lemma 4.2 of \cite{BH}]\label{lem:Fourier-IES} 
        Let $M \ge 1$, and let $\Psi \in C^{\infty}_c(\mathbb{R}_{\ge 0})$ be a function such that for all integer $j \ge 0$, $y^{\max(j, 1/2)}\Psi^{(j)}(y) \ll_j M^{j}$. If $\theta(s)=\pi^{-s}\Gamma(s)\zeta(2s)$, then for all $y > 0$
		\begin{equation}\label{eq:fourier,0}
			a_{\Psi,0}(y) = \frac{3}{\pi}\widehat{\Psi}(-1) + \frac{1}{2\pi}\int_\mathbb{R} \widehat{\Psi}(-1/2-it)\Big(y^{1/2+it} + \frac{\theta(1/2-it)}{\theta(1/2+it)}y^{1/2-it}\Big)\, dt.
		\end{equation}
            For any integer $l \neq 0$, if $\varepsilon > 0$ and $A \ge 0$, then for all $y > 0$
        \begin{equation}\label{eq:fourier,neq0}
                a_{\Psi,l}(y) \ll_{A}  \tau(|l|)\sqrt{y} M^{2/3+\varepsilon}(M/|ly|)^A(1+{1}/{|ly|})^\varepsilon.
		\end{equation}
	\end{lem}
 \begin{lem}\label{lem:H_cor_main} 
		Recall the definitions of $M_f(x)$ and $q_{\diamond}$ in Lemma \ref{lem:H_cor_body}, the notation in Section \ref{subsec:poincare}, and the bound in \eqref{eq:derivative-Psi_m}.  
        If $\varepsilon \in (0,1)$, then we have that
		\[
		\langle E(\cdot|\Psi_0), \overline{F_k}G_k \rangle_q  - \frac{3}{\pi}\widehat{\Psi}_0(-1)\mathds{1}_{f=g}  \ll M^{5/3+\varepsilon}q_\diamond^\varepsilon (\log kq)^\varepsilon M_f(kq)^{1/4}M_g(kq)^{1/4}.
		\]
	\end{lem}
        To prove Lemma~\ref{lem:H_cor_main}, we follow the method of \cite{H} and \cite{LMR}, as formulated in \cite[Lemma~4.4]{BH}. The key idea is to relate $\langle E(\cdot|\Psi_{0}), \overline{F_{k}} G_{k} \rangle_{q}$
        to an auxiliary integral $I(Y)$ involving a smoothed incomplete Eisenstein series $E^{Y}(z| h)$. 
        We then show that $\langle E(\cdot|\Psi_{0}), \overline{F_{k}} G_{k} \rangle_{q}$ is asymptotic to $I(Y)Y^{-1}$ up to a controllable error. 
        Through the unfolding method, the integral $I(Y)$ decomposes into the main term $I_{0}(Y)$, arising from the zeroth term in the Fourier expansion of $E(z | \Psi_{0})$, and the error terms $\sum_{l \ne 0} I_{l}(Y)$, coming from the remaining terms in the expansion. After estimating these components, we then choose an optimal value of $Y$.

        Let $1 \le Y \le (\log kq)^5$ be chosen later. 
		Let $h$ be a fixed nonnegative smooth compactly supported function on $\mathbb{R}$. Assume $\supp h \subset [1,2], h^{(j)}(y) \ll_j 1$, and $\widehat{h}(-1) = \pi/3$. By integration by parts, $\widehat{h}(-s) \ll (1+|s|)^{-A}$ for any $s > 0$ and $A \ge 0$. Define the smoothed incomplete Eisenstein series
		\[
		E^Y(z|h) = \sum_{\gamma \in \Gamma_\infty \bs \Gamma_0(1)} h(Y\Imag(\gamma z)).
		\]
		By the Mellin inversion formula, we push the contour to $\Real(s) = 1/2$ and obtain 
	\begin{equation}\label{eq:EY-mellin}
			E^Y(z|h) = Y + \int_{(1/2)}\widehat{h}(-s)Y^s E(s, z) \,\frac{ds}{2\pi i}.
	\end{equation}
		
        Define 
		\[
		I(Y) := \langle E^Y(\cdot| h)E(\cdot|\Psi_0), \overline{F_k}G_k\rangle_q = \int_{(2)}\widehat{h}(-s)Y^s \langle E(s, \cdot)E(\cdot | \Psi_0), \overline{F_k}G_k \rangle_q \,\frac{ds}{2\pi i}.
		\]
        By  \eqref{eq:EY-mellin}, we obtain that \begin{multline}\label{eq:cor-IES-1}
				\frac{I(Y)}{Y} -  \langle E(\cdot | \Psi_0), \overline{F_k}G_k \rangle_q = \frac{1}{Y}\int_{-\infty}^{\infty}\widehat{h}(-\tfrac{1}{2}-it)Y^{\frac{1}{2}+it} \langle E(\tfrac{1}{2}+it, \cdot)E(\cdot | \Psi_0), \overline{F_k}G_k \rangle_q \,\frac{dt}{2\pi}.
		\end{multline}
  
 We note that $y^{1/2}\Psi_0(y) \ll 1$ due to \eqref{eq:derivative-Psi_m}. Using \eqref{eq:Huang-xu-bound}, the fact that $\Psi_0(y)$ is nonzero only  if $y \ge \sqrt{3}/{2}$, and \cite[Lemma 2.10]{Iw}, we have that
		\begin{equation}
                \label{eq:EE-bound}
			\begin{aligned}
				|E(z |\Psi_0)E(1/2+it, z)| \ll (1+|t|)^{3/8}.
			\end{aligned}
		\end{equation}
       By the inequality of arithmetic and geometric means and \eqref{eq:EE-bound}, we have that
	\begin{equation}\label{eq:trivial-bound}
		\langle E(1/2+it, \cdot) E(\cdot | \Psi_0), \overline{F_k}G_k \rangle_q \ll (1+|t|)^{3/8}(\langle 1, |F_k|^2 \rangle_q+\langle 1, |G_k|^2 \rangle_q) \ll (1+|t|)^{3/8}.  
		\end{equation}
		Thus, inserting \eqref{eq:trivial-bound} into \eqref{eq:cor-IES-1} and using the rapid decay of $\widehat{h}(-1/2-it)$, we have that 
        \begin{equation}\label{eq:cor-IES-2}
		  I(Y)Y^{-1} -  \langle E(\cdot | \Psi_0), \overline{F_k}G_k \rangle_q 
            \ll Y^{-1/2},
		\end{equation}
        which implies that  we need to estimate $I(Y)$ in order to estimate $\langle E(\cdot | \Psi_0), \overline{F_k}G_k \rangle_q$. 

        Similarly to Lemma \ref{lem:unfolding}, we can apply the unfolding method with $I(Y)$ to obtain 
        \begin{equation}\label{eq:I(Y)-unfolding}
            I(Y) =
                \sum_{\mathfrak{a}_j \in \mathcal{C}} \int_{y_j=0}^{\infty} h(Yw_jy_j) \int_{x_j =0}^1 E(w_jz_j|\Psi_0) F_k(z)\overline{G_k(z)}y_j^{-2}dx_jdy_j.
        \end{equation}
        Consider $I(Y)$ when $f =g$. Since $|F_k(z)|^2 = |\rho_f(1)|^2|f|_k \sigma_j(z_j)|^2 y_j^k$ as in Section \ref{subsec:newforms}, we have
        \begin{equation}\label{eq:cor-IES-3}
			\begin{aligned}
				I(Y)
				= |\rho_f(1)|^2 \sum_{\mathfrak{a}_j \in \mathcal{C}} \int_{y_j=0}^{\infty} h(Yw_jy_j) \int_{x_j =0}^1 E(w_jz_j|\Psi_0) |f|_k \sigma_j(z_j)|^2 y_j^{k-2} dx_jdy_j.
			\end{aligned}
		\end{equation}	
        In \eqref{eq:cor-IES-3}, we write $E(w_jz_j|\Psi_0) = \sum_{l \in \mathbb{Z}} a_{\Psi_0, l}(w_jy_j)e(lw_jx_j)$, expand $f|_k \sigma_j(z_j)$ as in \eqref{eq:fourier-zj}, and integrate \eqref{eq:cor-IES-3} over $x_j$. Then it follows that $I(Y) = \sum_{l \in \mathbb{Z}} I_l(Y)$, where \begin{equation}\label{eq:I_l(y)}
        \begin{aligned}
			I_0(Y) &=  |\rho_f(1)|^2 \sum_{\mathfrak{a}_j \in \mathcal{C}} \sum_{n \ge 1} \lambda_j(n)^2n^{k-1} \int_{0}^{\infty} h(Yw_jy_j)a_{\Psi_0, 0}(w_jy_j) \exp(-4\pi n y_j)y_j^{k-2}dy_j;\\
			I_l(Y) &=  |\rho_f(1)|^2   \sum_{\mathfrak{a}_j \in \mathcal{C}}\,\,\sum_{\substack{n_1, n_2 = n_1+lw_j \ge 1}} \lambda_j(n_1)\lambda_j(n_2)(n_1n_2)^{\frac{k-1}{2}}\\
		&\qquad	\cdot \int_{0}^{\infty} h(Yw_jy_j) a_{\Psi_0, l}(w_jy_j) \exp(-2\pi(n_1+n_2)y_j)y_j^{k-2} dy_j, \qquad l \neq 0.
        \end{aligned}
        \end{equation}
        When $f \neq g$, $I_0(Y)$ and $I_l(Y)$ can be simplified similarly to the remark of Theorem \ref{rek:unfolding}.
  
        \begin{lem} 
        \label{lem:I_l(Y)}
            Recall the notation in Lemma \ref{lem:H_cor_main} and  \eqref{eq:I_l(y)}. If $\varepsilon \in (0,1)$, then
            \[
            \sum_{l\in \mathbb{Z}\backslash\{0\}} I_l(Y)Y^{-1} \ll M^{5/3+\varepsilon}Y^{1/2+\varepsilon}q_\diamond^\varepsilon (\log kq)^\varepsilon M_f(kq)^{1/2}M_g(kq)^{1/2}.
            \]
        \end{lem}
        \begin{proof} Assume $f=g$. We follow the idea in Lemma \ref{lem:H_cor_body}. By applying  the Mellin inversion formula for $h(Yw_jy_j)$, Stirling's formula, and the rapid decay of $\widehat{h}(-s)$, we obtain that
	\begin{equation*}
            \int_{0}^{\infty} h(Yw_jy_j) \exp(-2\pi(n_1+n_2)y_j)y_j^{k-2}dy_j
            \ll_{A} \frac{\Gamma(k-1)}{(4\pi)^{k-1}} \Big(\frac{1}{\frac{n_1+n_2}{2}}\Big)^{k-1} \max\Big(1, \frac{\max(n_1,n_2)}{Yw_jk}\Big)^{-A}. 
	\end{equation*}
        By \eqref{eq:fourier,neq0}, the above bound, and the fact that $\supp h \subset [1,2]$, the integral in \eqref{eq:I_l(y)} is
    \begin{equation}\label{eq:I_l(y)-mellin}
    \ll_{A}
        M^{\frac{2}{3}+{\frac{\varepsilon}{2}} }Y^{-\frac{1}{2}+{\frac{\varepsilon}{2}} }\tau(|l|)({MY}/{|l|})^A \frac{\Gamma(k-1)}{(4\pi)^{k-1}} \Big(\frac{1}{\frac{n_1+n_2}{2}}\Big)^{k-1} \max\Big(1, \frac{\max(n_1,n_2)}{Yw_jk}\Big)^{-A}.
    \end{equation} 
	
        Applying \eqref{eq:I_l(y)-mellin}  in \eqref{eq:I_l(y)},  we have that
        \begin{multline}\label{eq:I_l(y)-combined}
			I_l(Y)	\ll_{A} 
            |\rho_f(1)|^2 \frac{\Gamma(k-1)}{(4\pi)^{k-1}}  M^{\frac{2}{3}+{\frac{\varepsilon}{2}} }Y^{-\frac{1}{2}+{\frac{\varepsilon}{2}} }\tau(|l|)(MY/|l|)^A\\ \cdot \sum_{\mathfrak{a}_j \in \mathcal{C}}\sum_{\substack{n_1 \ge 1\\ n_2 = n_1+lw_j \ge 1}} \lambda_j(n_1)\lambda_j(n_2)\Big(\frac{\sqrt{n_1n_2}}{\frac{n_1+n_2}{2}}\Big)^{k-1}\max\Big(1, \frac{\max(n_1,n_2)}{Yw_jk}\Big)^{-A}.
		\end{multline} 
   Using \eqref{eq:rho_f(1)} and \eqref{eq:regrouping}, and applying Proposition \ref{prop:shiftedcon} dyadically, we obtain that
		\begin{multline*}
				I_l(Y) \ll_{A} \frac{M^{\frac{2}{3}+{\frac{\varepsilon}{2}} }Y^{-\frac{1}{2}+{\frac{\varepsilon}{2}} }\tau(|l|)({MY}/{|l|})^A}{kq L(1, \ad f)} q_\diamond^{\varepsilon} (\log\log(e^eq))^{O(1)}\\
				\cdot \sum_{c|q}\frac{\varphi((c,q/c))[q/c^2,1]}{\log([q/c^2,1]k)^{2-{\frac{\varepsilon}{2}} }}Yk \prod_{p \le Ykq}(1+2|\lambda_f(p)|/p).
		\end{multline*}
           By Proposition \ref{prop:shiftedcon2} and observing that the contribution from the product when $kq \le p \le Ykq$ is trivial since we assume $1\le Y \le (\log kq)^5$, we then have
           
            \begin{equation*}
			\begin{aligned}
			I_l(Y) 	&\ll_{A} \frac{M^{\frac{2}{3}+{\frac{\varepsilon}{2}} }Y^{\frac{1}{2}+{\frac{\varepsilon}{2} }}\tau(|l|)\big(MY/|l|\big)^A}{L(1, \ad f)} \frac{q_\diamond^{\varepsilon} (\log\log(e^eq))^{O(1)}}{(\log kq)^{2-{\frac{\varepsilon}{2}} }} \prod_{p \le kq}(1+2|\lambda_f(p)|/p).
			\end{aligned}
		\end{equation*}
        
       By the definition of $M_f(x)$ in Lemma \ref{lem:H_cor_body} and taking $A = 1+\varepsilon/2$, we obtain the  result.
  
    If $f \neq g$, we proceed similarly, except that $I_l(Y)$ can be simplified just as in the remark of Lemma \ref{rek:unfolding}.
        \end{proof}
        
        \begin{lem}
        \label{lem:I_0(Y)}
            Recall the notation in Lemma \ref{lem:H_cor_main} and \eqref{eq:I_l(y)}. If $\varepsilon \in (0,1)$, then we have that
            \[
           I_0(Y)Y^{-1} - \frac{3}{\pi}\widehat{\Psi}_0(-1)\mathds{1}_{f=g} \ll M^{1+\varepsilon}Y^{-1/2}(\log kQ)^\varepsilon.
            \]
        \end{lem}
        \begin{proof} By \eqref{eq:I(Y)-unfolding} and the expression of $a_{\Psi_0,0}(y)$ in \eqref{eq:fourier,0}, we have that
        \begin{align}
        \label{eq:IES-1}
		I_0(Y)
        &= \sum_{\mathfrak{a}_j \in \mathcal{C}} \int_{y_j=0}^{\infty}\int_{x_j =0}^1 \mathscr{I}(w_j y_j)h(Yw_jy_j) F_k(z)\overline{G_k(z)}  {y_j^{-2}}dx_jdy_j,
        \end{align}
        where
        \begin{align*}
        \mathscr{I}(w_jy_j) &:=   \frac{3}{\pi}\widehat{\Psi}_0(-1)
        + \frac{1}{2\pi}\int_\mathbb{R} \widehat{\Psi}_0(-\tfrac{1}{2}-it)
        \Big((w_jy_j)^{\frac{1}{2}+it} + \frac{\theta(\frac{1}{2}-it)}{\theta(\frac{1}{2}+it)}(w_jy_j)^{\frac{1}{2}-it}\Big)\, dt.
	\end{align*}

        By \eqref{eq:EY-mellin}, the contribution to $I_0(Y)$ from the term $\frac{3}{\pi}\widehat{\Psi}_0(-1)$ in $\mathscr{I}(w_jy_j)$  equals
        \begin{equation}
        \begin{aligned}
        \label{eq:cor-IES-5}
            \frac{3}{\pi}\widehat{\Psi}_0(-1) \langle E^Y(z|h), \overline{F_k}G_k\rangle_q 
            &=\frac{3}{\pi}\widehat{\Psi}_0(-1)\Big(Y \mathds{1}_{f=g}+ \int_{(1/2)}\widehat{h}(-s)Y^s \langle E(s,\cdot), \overline{F_k}G_k\rangle_q \,\frac{ds}{2\pi i}\Big)\\
            &= \frac{3}{\pi}\widehat{\Psi}_0(-1)Y\mathds{1}_{f=g} + O\Big((q/\sqrt{Q})^{-1/2+\varepsilon}Y^{1/2}(\log kQ)^{\varepsilon}\Big),
            \end{aligned}
		\end{equation}
  upon using \eqref{eq:sound_Eis},  
  the rapid decay of $\widehat{h}(-s)$, and \eqref{eq:l(1,adf)GHL}.

Since $\supp h \subset [1,2]$,  $h(Yw_jy_j)$ is nonzero only if $w_jy_j \asymp Y^{-1}$. Using \eqref{eq:rapiddecay} with $j = 1 +\varepsilon$, we have that the contribution to $I_0(Y)$ from the integral in $\mathscr{I}(w_jy_j)$ is
\begin{equation}
\label{eq:IES-2}
    \begin{aligned}
        &\ll (M^{1+\varepsilon}Y^{-1/2})\sum_{\mathfrak{a}_j \in \mathcal{C}} \int_{y_j=0}^{\infty} h(Yw_jy_j)\int_{x_j =0}^1
        |F_k(z)||G_k(z)|  y_j^{-2}dx_jdy_j.
    \end{aligned}
\end{equation}
Using $|F_k(z)||G_k(z)| \ll |F_k(z)|^2+ |G_k(z)|^2$ and \eqref{eq:cor-IES-5} with $f=g$,
we have that \eqref{eq:IES-2} is 
\begin{equation}
\label{eq:IES-3}
   \ll M^{1+\varepsilon}Y^{-1/2}( Y + (q/\sqrt{Q})^{-1/2+\varepsilon}Y^{1/2}(\log kQ)^{\varepsilon}).
\end{equation}

Inserting \eqref{eq:cor-IES-5} and \eqref{eq:IES-3} in \eqref{eq:IES-1}, we obtain that 
\begin{multline*}
				I_0(Y) = \frac{3}{\pi}\widehat{\Psi}_0(-1)\big(Y\mathds{1}_{f=g}+ O\big((q/\sqrt{Q})^{-1/2+\varepsilon}Y^{1/2}(\log kQ)^{\varepsilon}\big)\big)\\
			  + O\big(M^{1+\varepsilon}Y^{-1/2} \big(Y+ (q/\sqrt{Q})^{-1/2+\varepsilon}Y^{1/2}(\log kQ)^{\varepsilon}\big)\big).
		\end{multline*}
            Hence, the lemma follows from the facts that $\frac{3}{\pi}\widehat{\Psi}_0(-1) \ll 1$ and $Y \ge 1$. 
            \end{proof}

        \begin{proof}[Proof of Lemma \ref{lem:H_cor_main}]
		Applying Lemmas \ref{lem:I_l(Y)} and \ref{lem:I_0(Y)} in \eqref{eq:cor-IES-2}, we obtain that
		\begin{multline*}
				\langle E(\cdot|\Psi_0), \overline{F_k}G_k\rangle_q - \frac{3}{\pi}\widehat{\Psi}_0(-1)\mathds{1}_{f=g}\\
                \ll M^{1+\varepsilon}Y^{-1/2}(\log kQ)^\varepsilon + M^{5/3+\varepsilon}Y^{1/2+\varepsilon}q_\diamond^\varepsilon  (\log kq)^\varepsilon M_f(kq)^{1/2}M_g(kq)^{1/2}+Y^{-1/2}.
		\end{multline*}
	   The lemma follows upon choosing $Y = M_f(kq)^{-1/2}M_g(kq)^{-1/2}$ and noting that $Q \le q^2$.
        \end{proof}
 
	\begin{proof}[Proof of Proposition \ref{prop:Holow_cor}]

		Consider the factor $M_f(kq)$. When applying Lemma \ref{lem:L(1,adf)} and the facts that $\prod_{p\le x}(1-\delta/p) \asymp (\log x)^{-\delta}$ and  $\log(kq)/\log(kQ) \ll (q/\sqrt{Q})^{\varepsilon}$, we have that
		\begin{equation}\label{eq:Mf(kq)}
			M_f(kq) \ll \prod_{p\le k\sqrt{Q}} \Big(1-\frac{(|\lambda_f(p)|-1)^2}{p}\Big)\Big(\frac{q}{\sqrt{Q}}\Big)^{\varepsilon}(\log kq)^\varepsilon.
		\end{equation}  
 		Putting together Lemmas \ref{lem:H_cor_tail}, \ref{lem:H_cor_body}, and \ref{lem:H_cor_main} in \eqref{eq:poincare2} and using \eqref{eq:Mf(kq)} and the fact that $q_\diamond^{\varepsilon} \ll (q/\sqrt{Q})^\varepsilon$, we obtain the proposition as desired. 
        \end{proof}
        

\appendix
\section{The \texorpdfstring{$\mathrm{GL}_2\times\mathrm{GL}_3$}{GL23} approximate functional equation and QUE}
\label{appendix}
\begin{center} {\sc by Jesse Thorner}
\end{center}
\vspace{3mm}
Let $f$ be a holomorphic cuspidal Hecke eigenform of even weight $k\geq 2$ and $\phi$ be a Hecke--Maa{\ss} cusp form, both of level 1.  Each corresponds with a certain unitary cuspidal automorphic representation of $\mathrm{GL}_2(\mathbb{A}_{\mathbb{Q}})$.  Under the stated hypotheses, the Hecke eigenvalues $\lambda_{f}(n)$ and $\lambda_{\phi}(n)$ of $f$ and $\phi$ (respectively) are all real and multiplicative.

Let $F_k(x+iy) = \rho_f(1) y^{k/2}f(x+iy)$, where $\rho_f(1)$ is a normalizing constant so that if $\langle \cdot,\cdot\rangle$ is the Petersson inner product for $Y_0(1)$, then $\langle 1,|F_k|^2\rangle=1$. 
Iwaniec \cite[Section 11]{Iwaniec} and Lester, Matom{\"a}ki, and Radziwi{\l}{\l} \cite[Section 4]{LMR} describe a method for showing that $\lim_{k\to\infty}\langle \phi,|F_k|^2\rangle=0$ without a weak subconvexity bound for $L(\tfrac{1}{2},\mathrm{Ad}\,f\times\phi)$. This appendix revisits their proofs and identifies a minor gap. A careful correction shows that the convergence of $\langle \phi, |F_k|^2\rangle$ occurs at a slower rate than was previously claimed.  For simplicity, we suppress the $\phi$-dependence.

It is proved in \cite[(3.2) and (11.1)]{Iwaniec} and  \cite[(4.2)]{LMR} that
\begin{equation}
\label{eqn:start_AFE}
|\langle\phi,|F_k|^2\rangle|\ll_{A,\phi,\varepsilon}\frac{(\log\log k)^3}{\sqrt{k}}|L(\tfrac{1}{2},\mathrm{Ad}\,f\times\phi)|^{\frac{1}{2}}\prod_{p\leq k}\Big(1-\frac{\lambda_f(p^2)}{p}\Big)+\frac{1}{(\log k)^{A}}.
\end{equation}
A bound for $L(\frac{1}{2},\mathrm{Ad}\,f\times\phi)$ essentially of the form
\begin{equation}
\label{eqn:false}
|L(\tfrac{1}{2},\mathrm{Ad}\,f\times\phi)|\ll_{A,\phi,\varepsilon}\sum_{n\leq k^2(\log k)^{\varepsilon}}\frac{|\lambda_{f}(n^2)\lambda_{\phi}(n)|}{\sqrt{n}}+\frac{k}{(\log k)^A}
\end{equation}
(a weak form of the approximate functional equation) was asserted without proof in \cite[Section 11.1]{Iwaniec} and \cite[(4.3)]{LMR}.  These appear to neglect Theorem \ref{thm:1} below.
\begin{thm}
\label{thm:1}
Let $f$ be a holomorphic cuspidal Hecke eigenform and $\phi$ be a Hecke--Maa{\ss} newform, each of level $1$.  There exists an absolute constant $0<\omega\leq \frac{1}{10}$ and a function $H(s)$ such that in the region $\mathrm{Re}(s)\geq \frac{1}{2}-\omega$, $H(s)$ is holomorphic, $H(s)\asymp 1$, and
\[
L(s,\mathrm{Ad}\, f\times \phi)=H(s) L(2s,\mathrm{Ad}\, f) L(2s,\mathrm{Ad}\, \phi) \sum_{p\mid n\implies p\ge 19}\frac{\lambda_{f}(n^2)\lambda_{\phi}(n)}{n^{s}}.
\]
\end{thm}
\begin{proof}
By multiplicativity of the Hecke eigenvalues, it suffices to work locally.  Let $p$ be prime.  For any $\pi\in\mathfrak{F}_2$ of conductor $1$, we have that
\begin{align*}
L_p(s,\pi) &= \prod_{\ell\in\{-1,1\}}\frac{1}{1-\alpha_{\pi}(p)^{\ell} p^{-s}}=\sum_{j=0}^{\infty}\frac{\lambda_{\pi}(p^j)}{p^{js}},\\
L_p(s,\mathrm{Ad}\, \pi) &= \prod_{\ell\in\{-1,0,1\}}\frac{1}{1-\alpha_{\pi}(p)^{2\ell} p^{-s}}=\sum_{j=0}^{\infty}\frac{\lambda_{\mathrm{Ad}\, \pi}(p^j)}{p^{js}}.
\end{align*}
We write $\alpha_p = \alpha_{f}(p)$, $\beta_p = \alpha_{\phi}(p)$, and
\[
L_p(s,\mathrm{Ad}\, f \times \phi) = \prod_{\ell \in\{-1,0,1\}}\,\prod_{\ell'\in\{-1,1\}}\frac{1}{1-\alpha_p^{2\ell} \beta_p^{\ell'} p^{-s}}=\sum_{j=0}^{\infty}\frac{\lambda_{\mathrm{Ad}\, f\times\phi}(p^j)}{p^{js}}.
\]
Note that $\lambda_f(p^j)= \sum_{r=0}^j \alpha_p^{2r-j}$ and $\lambda_{\phi}(p^j)= \sum_{r=0}^j \beta_p^{2r-j}$.  By \eqref{eqn:GRC2} and the Ramanujan conjecture for $\pi$, we have the bounds
\begin{equation}
\label{eqn:GRC2}
|\alpha_p|=1,\qquad \max\{|\beta_p|,|\beta_p|^{-1}\}\leq p^{\frac{7}{64}}.
\end{equation}

Consider the rational function
\begin{equation}
\label{eqn:aux_rational}
\frac{1+(\beta_p+\beta_p^{-1})p^{-s}-(\alpha_p^2+1+\alpha_p^{-2})p^{-2s}}{(1-\alpha_p^2 \beta_p p^{-s})(1-\alpha_p^2 \beta_p^{-1}p^{-s})(1-\alpha_p^{-2} \beta_p p^{-s})(1-\alpha_p^{-2} \beta_p^{-1}p^{-s})}.
\end{equation}
On one hand, we can expand \eqref{eqn:aux_rational} as a polynomial in $p^{-s}$ times a product of four geometric sums.  On the other hand, we can identify \eqref{eqn:aux_rational} as a product of local $L$-functions.  Identifying these two expressions, we conclude that if $\mathrm{Re}(s)>7/64$, then
\begin{equation}
\label{eqn:key}
L_p(s,\mathrm{Ad}\,f\times\phi)=\frac{L_p(s,\phi)}{1+\lambda_{\phi}(p)p^{-s}-\lambda_{f}(p^2)p^{-2s}}\sum_{j=0}^{\infty}\frac{\lambda_{f}(p^{2j})\lambda_{\phi}(p^j)}{p^{js}},
\end{equation}
provided that the denominator for the leading factor on the right-hand side is nonzero.  We verify that if $p\geq 19$ and $\mathrm{Re}(s)\geq \frac{2}{5}$, then $|1+\lambda_{\phi}(p)p^{-s}-\lambda_{f}(p^2)p^{-2s}|>0.06$ using \eqref{eqn:GRC2} and Mathematica 12.  For such $p$ and $s$, we find that
\begin{equation}
\label{eqn:G_def}
\frac{L_p(s,\phi)}{1+\lambda_{\phi}(p)p^{-s}-\lambda_{f}(p^2)p^{-2s}}=1+(\lambda_{f}(p^2)+\lambda_{\phi}(p^2))p^{-2s} +\cdots,
\end{equation}
which agrees with the beginning of the expansion $L_p(2s,\mathrm{Ad}\,f)L_{p}(2s,\mathrm{Ad}\,\phi)$ up to an error of $O(p^{-3s})$. This leads us to define
\[
H_p(s) = \frac{L_p(s,\phi)}{(1+\lambda_{\phi}(p)p^{-s}-\lambda_{f}(p^2)p^{-2s})L_p(2s,\mathrm{Ad}\, f)L_p(2s,\mathrm{Ad}\,\phi)}.
\]
By our above numerical calculation and \eqref{eqn:GRC2}, we know that if $p\geq 19$,  then in the region $\mathrm{Re}(s)\geq \frac{2}{5}$, $H_p(s)$ is holomorphic and
\begin{align*}
L_p(s,\mathrm{Ad}\,f\times\phi) &= H_p(s)L_p(2s,\mathrm{Ad}\, f)L_p(2s,\mathrm{Ad}\,\phi)\sum_{j=0}^{\infty}\frac{\lambda_{f}(p^{2j})\lambda_{\phi}(p^j)}{p^{js}},\\
H_p(s)&=1-\lambda_{f}(p)^2\lambda_{\phi}(p)p^{-3s}+(\lambda_{f}(p)^2(\lambda_{\phi}(p)^2+1)-2)p^{-4s}+\cdots.
\end{align*}

Since the coefficients of $p^{-s}$ and $p^{-2s}$ in the expansion of $H_p(s)$ are zero, \eqref{eqn:GRC2} implies that there exists an absolute constant $0<\omega\leq\frac{1}{10}$ such that if $\mathrm{Re}(s)\geq \frac{1}{2}-\omega$, then the functions
\[
\mathcal{H}_1(s)=\prod_{p\geq 19}H_p(s),\qquad \mathcal{H}_2(s) = \prod_{p<19}\frac{L_p(s,\mathrm{Ad}\,f\times\phi)}{L_p(2s,\mathrm{Ad}\,f)L_p(2s,\mathrm{Ad}\,\phi)}
\]
are holomorphic and $\asymp 1$.  To finish, we choose $H(s) = \mathcal{H}_1(s)\mathcal{H}_2(s)$.
\end{proof}

Let $\kappa$ be the root number of $L(s,\mathrm{Ad}\, f \times\phi)$, $L_{\infty}(s,\mathrm{Ad}\,f\times \phi)$ be the gamma factor, and
\begin{equation}
\label{eqn:choices}
A > 2,\qquad 0<\varepsilon<1,\qquad B = \frac{2A+6}{\varepsilon}+\frac{1}{2}>3.
\end{equation}
Using the functional equation for $L(s,\mathrm{Ad}\, f \times \phi)$ and Theorem \ref{thm:1}, we derive the identity
\begin{align}
\label{eqn:start}
&L(\tfrac{1}{2},\mathrm{Ad}\, f\times \phi) = (1+\kappa)\sum_{p\mid n\implies p \geq 19}\frac{\lambda_f(n^2)\lambda_{\phi}(n)}{\sqrt{n}}W(n),\\
&W(n) = \frac{1}{2\pi i}\int_{B-i\infty}^{B+i\infty}H(\tfrac{1}{2}+s)L(1+2s,\mathrm{Ad}\, f)L(1+2s,\mathrm{Ad}\, \phi)\frac{L_{\infty}(\tfrac{1}{2}+s,\mathrm{Ad}\, f \times \phi)}{L_{\infty}(\tfrac{1}{2},\mathrm{Ad}\, f \times \phi)}\frac{e^{s^2}}{n^s}\frac{ds}{s}.\notag
\end{align}
\begin{lem}
\label{lem:W_bounds}
With the notation in Theorem \ref{thm:1}, \eqref{eqn:choices}, and \eqref{eqn:start}, there holds
\[
W(n)\ll_{A,\phi,\varepsilon}\begin{cases}
	(k^2/n)^B&\mbox{if $n>k^2(\log k)^{\varepsilon}$,}\\
	 (\log k)^{\varepsilon} \prod_{p \le k} \big( 1+\frac{|\lambda_f(p^2)|}{p}\big) &\mbox{if $n\leq k^2(\log k)^{\varepsilon}$.}
\end{cases}
\]
\end{lem}
\begin{proof}
Recall \eqref{eqn:choices} and \eqref{eqn:start}.  Write $s=\sigma+it$.
Theorem \ref{thm:1} implies that $H(\tfrac{1}{2}+s)\ll 1$. 
When $\sigma=1/(2\log k)$, using the Euler product of $L(s, \ad f)$ and GRC for $f$, we have 
\[
|L(1+2s,\ad\,f)|  \le \prod_{p} \Big( 1+\frac{|\lambda_{\ad f}(p)|}{p^{1+2\sigma}} + \frac{|\lambda_{\ad f}(p^2)|}{p^{2(1+2\sigma)}}+\cdots\Big) \ll \prod_{p} \Big( 1+\frac{|\lambda_{\ad f}(p)|}{p^{1+2\sigma}}\Big).
\]
By the prime number theorem and partial summation, it follows that 
\[
\prod_{p \ge k} \Big( 1+\frac{|\lambda_{\ad f}(p)|}{p^{1+2\sigma}}\Big) \ll 1.
\]
Thus, if we proceed as in \cite[Theorem 2]{L} for $\phi$ and use the fact that $\lambda_{\ad f}(p) = \lambda_f(p^2)$, then
\[
|L(1+2s,\ad\,f)L(1+2s,\ad\,\phi)|\ll_{\phi,\varepsilon}\begin{cases}
	1&\mbox{if $\sigma=B$,}\\
	(|t|+1)^{\varepsilon}\prod_{p \le k} \big( 1+\frac{|\lambda_f(p^2)|}{p}\big) &\mbox{if $\sigma=\frac{1}{2\log k}$.}
\end{cases}
\]
Using Stirling's formula, we bound the integral defining $W(n)$ on the line $\sigma=B$ (when $n>k^2(\log k)^{\varepsilon}$) or on the line $\sigma=1/(2\log k)$ (when $n\leq k^2(\log k)^{\varepsilon}$).
\end{proof}
\begin{rek}
In our application of the Lemma \ref{lem:W_bounds}, we will sum $\lambda_f(n^2)\lambda_{\phi}(n)$ over $n\leq k^2(\log k)^{\varepsilon}$ instead of $n\leq k^{2+\varepsilon}$ because our result is sensitive to powers of $\log k$.  
\end{rek}

Theorem \ref{thm:1} and Lemma \ref{lem:W_bounds} yield the following bound for $L(\frac{1}{2},\mathrm{Ad}\,f\times\phi)$.
\begin{thm}
\label{thm:AFE}
If $f$ is a holomorphic cuspidal Hecke eigenform of even weight $k\geq 2$ and level $1$, $\phi$ is a Hecke--Maa{\ss} cusp form of level $1$, $A>2$, and $\varepsilon>0$, then
\begin{align*}
L(\tfrac{1}{2},\mathrm{Ad}\,f\times\phi)&=(1+\kappa)\sum_{\substack{n\leq k^2(\log k)^{\varepsilon} \\ p\mid n\implies p\geq 19}}\frac{\lambda_{f}(n^2)\lambda_{\phi}(n)}{\sqrt{n}}W(n)+O_{A,\phi,\varepsilon}\Big(\frac{k}{(\log k)^{2A+2}}\Big)\\
&\ll_{A,\phi,\varepsilon}    (\log k)^{\varepsilon} \Big(\prod_{p \le k} \Big( 1+\frac{|\lambda_{f}(p^2)|}{p}\Big)\Big)\sum_{n\leq k^2(\log k)^{\varepsilon}}\frac{|\lambda_{f}(n^2)\lambda_{\phi}(n)|}{\sqrt{n}} +\frac{k}{(\log k)^{2A+2}}. 
\end{align*}
\end{thm}

\begin{proof}
We estimate the tail of the sum in \eqref{eqn:start} using Lemma \ref{lem:W_bounds}, the Cauchy--Schwarz inequality, Rankin--Selberg theory, and \eqref{eqn:choices}.
\end{proof}

We now apply Theorem \ref{thm:AFE} to \eqref{eqn:start_AFE}.

\begin{thm}
\label{thm:inner}
If $f$ is a holomorphic cuspidal Hecke eigenform of even weight $k\geq 2$ and level $1$, $\phi$ is a Hecke--Maa{\ss} cusp form of level $1$, $A>2$, and $\varepsilon>0$, then
\[
\langle \phi,|F_k|^2\rangle\ll_{A,\phi,\varepsilon}
(\log k)^{\varepsilon} \prod_{p\leq k}\Big(1-\frac{\lambda_f(p^2){-\frac{1}{2}|\lambda_f(p^2)|}+\frac{1}{4}(1-\lambda_f(p^2))^2}{p}\Big)+\frac{1}{(\log k)^A}.
\]	
\end{thm}
\begin{proof}
By Theorem \ref{thm:AFE} and the Cauchy--Schwarz inequality, $L(\frac{1}{2},\mathrm{Ad}\,f\times\phi)^{1/2}$ is
\begin{multline}
\label{eqn:inner1}
\ll_{A,\phi,\varepsilon}(\log k)^{\frac{\varepsilon}{2}} \Big(\prod_{p \le k} \Big( 1+\frac{\frac{1}{2}|\lambda_{f}(p^2)|}{p}\Big)\Big)\Big(\sum_{n\leq k^2(\log k)^{\varepsilon}}\frac{\lambda_f(n^2)^2}{\sqrt{n}}\Big)^{\frac{1}{4}}\Big(\sum_{n\leq k^2(\log k)^{\varepsilon}}\frac{\lambda_{\phi}(n)^2}{\sqrt{n}}\Big)^{\frac{1}{4}}
\\+\frac{\sqrt{k}}{(\log k)^{A+1}}.
\end{multline}
It follows from Rankin--Selberg theory and partial summation that \eqref{eqn:inner1} is
\begin{equation}
\label{eqn:inner2}
\ll_{A,\phi,\varepsilon} k^{\frac{1}{4}}  (\log k)^{\varepsilon} \Big(\prod_{p \le k} \Big( 1+\frac{\frac{1}{2}|\lambda_{f}(p^2)|}{p} \Big)\Big)\Big(\sum_{n\leq k^2(\log k)^{\varepsilon}}\frac{\lambda_{f}(n^2)^2}{\sqrt{n}}\Big)^{\frac{1}{4}}+\frac{{\sqrt{k}}}{(\log k)^{A+1}}.
\end{equation}
We apply Shiu's variant of the Brun--Titchmarsh theorem \cite[Theorem 1]{Shiu} (with partial summation) to conclude that \eqref{eqn:inner2} is
\begin{align*}
&\ll_{A,\phi,\varepsilon} \sqrt{k}{ (\log k)^{\varepsilon} \Big(\prod_{p \le k} \Big( 1+\frac{\frac{1}{2}|\lambda_{f}(p^2)|}{p} \Big)\Big)}\prod_{p\leq k^2(\log k)^{\varepsilon}}\Big(1-\frac{\frac{1}{4}(1-\lambda_f(p^2)^2)}{p}\Big)+\frac{\sqrt{k}}{(\log k)^{A+1}}.
\end{align*}
The product over $p\in[k+1,k^2(\log k)^{\varepsilon}]$ is $O_{\varepsilon}(1)$, so the preceding display is
\begin{align*}
\ll_{A,\phi,\varepsilon} \sqrt{k}{(\log k)^{\varepsilon} \prod_{p\leq k}\Big(1-\frac{\frac{1}{4}(1-\lambda_f(p^2)^2)-\frac{1}{2}|\lambda_{f}(p^2)|}{p}\Big)}+\frac{\sqrt{k}}{(\log k)^{A+1}}.
\end{align*}
The desired result now follows from Lemma \ref{lem:L(1,adf)}, \eqref{eq:l(1,adf)GHL}, and \eqref{eqn:start_AFE}.
\end{proof}
  
Recall the Hecke relations $\lambda_{f}(p^2)=\lambda_f(p)^2-1$. For a sufficiently large $A$, it follows from Theorem \ref{thm:inner} that  
\begin{equation}
\label{eqn:inner}
\begin{aligned}
\langle \phi,|F_k|^2\rangle &\ll_{A,\phi,\varepsilon} 
(\log k)^{\varepsilon}\prod_{p\leq k}\Big(1-\frac{ (\lambda_f(p)^2-1) -\frac{1}{2}|\lambda_f(p)^2-1|+\frac{1}{4}(1-\lambda_{f}(p^2)^2)}{p}\Big). 
\end{aligned}\hspace{-2mm}
\end{equation}
By a different method, Holowinsky (see \cite[(1.2)]{HS} or \cite[(1.14)]{Iwaniec}) proved that
\[
\langle \phi,|F_k|^2\rangle \ll_{\phi,\varepsilon}(\log k)^{\varepsilon}\prod_{p\leq k}\Big(1-\frac{\frac{1}{2}(|\lambda_f(p)|-1)^2}{p}\Big).
\]
Recall the fact that if $\xi\in(0,1)$ and $a,b>0$, then $\min\{a,b\}\leq a^{\xi}b^{1-\xi}$.  For $\xi\in(0,1)$ and $A$ be sufficiently large with respect to $\xi$, we have 
\[
\langle \phi,|F_k|^2\rangle \ll_{A,\phi,\varepsilon}(\log k)^{-\delta+\varepsilon},
\]
where
\[ 
\delta = \max_{0\leq\xi\leq 1}~\min_{0\leq \lambda\leq 2}\Big(\frac{\xi}{2}(\lambda-1)^2+ (1-\xi)\Big(\lambda^2-1-\frac{1}{2}|\lambda^2-1|+\frac{1}{4}(1-(\lambda^2-1)^2)\Big)\Big) \approx 0.00348\ldots
\]
(See \cite[\S12]{Iwaniec} and \cite[\S4.3]{LMR}.) This corrected value of $\delta$ is smaller than the exponent $0.007359\ldots$ stated in \cite[\S12]{Iwaniec}.

\bibliographystyle{abbrv}
\bibliography{QUE}
\end{document}